\documentclass[11pt,reqno]{amsart}

\usepackage{latexsym}

\RequirePackage[OT1]{fontenc}
\RequirePackage{amsmath}
\usepackage{hyperref}
\usepackage{exscale}
\usepackage{amssymb,amscd,amsthm,mathtools}
\usepackage{color}
\usepackage{graphicx}
\usepackage{mathrsfs}

\usepackage{enumitem}       
\setlist[enumerate,1]{leftmargin=1cm}
\setitemize{leftmargin=1cm}

\DeclareGraphicsExtensions{.eps, .ps}

\theoremstyle{plain}

\newtheorem{theorem}{Theorem}
\newtheorem{proposition}{Proposition}

\newtheorem{lemma}[proposition]{Lemma}

\theoremstyle{definition}
\newtheorem{definition}{Definition}
\newtheorem{example}{Example}

\theoremstyle{remark}

\renewcommand{\to}{\rightarrow}

\newcommand{\BN}{\mathbb{N}}
\newcommand{\BR}{\mathbb{R}}
\newcommand{\BZ}{\mathbb{Z}}

\newcommand{\cA}{\mathcal{A}}

\newcommand{\cH}{\mathcal{H}}
\newcommand{\cI}{\mathcal{I}}

\newcommand{\cS}{\mathcal{S}}
\newcommand{\cT}{\mathcal{T}}

\def\cf{\mathbf{1}}
\def\Be{\mathbf{e}}

\newcommand{\T}{\mathcal{T}}

\def\scT{\mathscr{T}}
\newcommand{\scS}{\mathscr{S}}

\renewcommand{\iff}{\Leftrightarrow}

\newcommand{\Span}{\textsc{span}}

\def\tail{\textnormal{tail}}
\def\cl{\textsc{cl}}

\newcommand{\Restrict}[2]{#1\big|_{#2}}
\newcommand{\restrict}[2]{#1|_{#2}}
\newcommand{\fringe}[2]{F_{#2}(#1)}
\newcommand{\bush}[2]{B_{#2}(#1)}

\newcommand{\from}{\leftarrow}

\title{Mass-structure of weighted real trees}

\author{Noah Forman}
\thanks{Research supported in part by EPSRC grant EP/K029797/1 and NSF grant DMS-1444084}

\date{\today}

\begin{document}
\begin{abstract}
 Rooted, weighted continuum random trees are used to describe limits of sequences of random discrete trees. Formally, they are random quadruples $(\T,d,r,p)$, where $(\T,d)$ is a tree-like metric space, $r\in\T$ is a distinguished \emph{root}, and $p$ is a probability measure on this space. The underlying branching structure is carried implicitly in the metric $d$. We explore various ways of describing the interaction between branching structure and mass in $(\T,d,r,p)$ in a way that depends on $d$ only by way of this branching structure. We introduce a notion of mass-structure equivalence and show that two rooted, weighted $\BR$-trees are equivalent in this sense if and only if the discrete hierarchies derived by i.i.d.\ sampling from their weights, in a manner analogous to Kingman's paintbox, have the same distribution. We introduce a family of trees, called ``interval partition trees'' that serve as representatives of mass-structure equivalence classes, and which naturally represent the laws of the aforementioned hierarchies.
\end{abstract}
\maketitle

\section{Introduction}\label{sect-intro}

This paper explores three ideas that we find to be closely related: a notion of ``mass-structural equivalence'' between rooted, weighted real trees; a family of such trees in which the metric is, in a sense, specified by the weight and underlying branching structure; and continuum random tree representations of exchangeable random hierarchies on $\BN$.

\begin{definition} \label{def:real_tree}
 A \emph{real tree} (\emph{$\BR$-tree}) is a complete, separable, bounded metric space $(\T,d)$ with the property that: (i) for each $x,y\in\T$, there is a unique non-self-intersecting path in $\T$ from $x$ to $y$, called a \emph{segment} $[[x,y]]_{\T}$, and (ii) each segment is isometric to a closed real interval. Some authors require that $\BR$-trees be compact, but we will not.
 
 
 A \emph{rooted, weighted $\BR$-tree} is a quadruple $(\T,d,r,p)$, where $(\T,d)$ is a $\BR$-tree, $r\in \T$ is a distinguished vertex called the \emph{root}, and $p$ is a probability distribution on $\T$ with respect to the Borel $\sigma$-algebra generated by $d$.
 
 We call two rooted, weighted $\BR$-trees \emph{isomorphic} if there exists a root- and weight-preserving isometry between them.
\end{definition}

$\BR$-trees have long been studied by topologists; see \cite{EvansStFleur,SempleSteel} for references. Random $\BR$-trees, called  \emph{continuum random trees} (\emph{CRTs}), were first studied by Aldous \cite{MR1085326,MR1207226}; also see \cite{EvansStFleur,MR2203728}. 
In particular, Aldous introduced the Brownian CRT, which arises as a scaling limit of various families of random discrete trees, including critical Galton-Watson trees conditioned on total progeny. The Brownian CRT is a random fractal in the sense that, if we decompose it around a suitably chosen random branch point, then the components are each distributed as scaled copies of a Brownian CRT and are conditionally independent given their sizes. Since Aldous's work, other authors have introduced similarly complex CRTs, such as the Stable CRTs \cite{DuquLeGall05,DuquLeGall02}.

We think of these CRTs as having complex underlying ``branching structures.'' Formally, the branching structure in a $\BR$-tree $(\T,d)$ is specified by the metric $d$. But for some applications, it may be of interest to describe this structure in a way that does not depend on quantifying distances. In this paper, we consider the interaction between branching structure and mass in rooted, weighted $\BR$-trees. In particular, we look at various representations of this interaction that do not depend on the metric, except by way of the underlying branching structure.

 For a rooted $\BR$-tree $(\T,d,r)$, a point $x\in \T$ is a \emph{branch point} if there exist three non-trivial segments with endpoint $x$ whose pairwise intersections all equal $\{x\}$. A point $x$ is a \emph{leaf} if it is an endpoint of every segment to which it belongs. The complement of the set of leaves is the \emph{skeleton} of the tree. The \emph{fringe subtree} of $(\T,d,r)$ rooted at $x$ is
 \begin{equation}
  \fringe{x}{\T}:= \{ y \in \T\colon x \in [[r,y]]_{\cT}\}.\label{fringe-subtree-defn}
 \end{equation}

\begin{definition}\label{def:special_pts}
 Consider a rooted, weighted $\BR$-tree $(\T,d,r,p)$. 
 The \emph{subtree spanned by (the closed support of) $p$} is 
 \begin{equation}
  \Span(p) := \bigcup_{x\in \textnormal{support}(p)}[[s,x]]_{\cT} .
 \end{equation}
 The \emph{special points} of $(\T,d,r,p)$ are:
 \begin{enumerate}[label=(\alph*), ref=(\alph*)]
  \item the locations of atoms of $p$,\label{item:special_atoms}
  \item the branch points of $\Span(p)$, and\label{item:special_BPs}
  \item the \emph{isolated leaves} of $\Span(p)$, by which we mean leaves of $\Span(p)$ that are not limit points of the branch points of $\Span(p)$.\label{item:special_leaves}
 \end{enumerate}
\end{definition}

\begin{definition}\label{def:mass_struct}
 Let $\scS_i$ denote the set of special points of a tree $(\T_i,d_i,r_i,p_i)$ for $i=1,2$. 
 A \emph{mass-structural isomophism} between these $\BR$-trees is a bijection $\phi\colon \scS_1\to\scS_2$ with the following properties.
 \begin{enumerate}[label=(\roman*), ref=(\roman*)]
  \item \emph{Mass preserving}. For every $x\in\scS_1$, 
  $p_1\big([[r_1,x]]_{\T_1}\big) = p_2\big([[r_2,\phi(x)]]_{\cT_2}\big)$, $p_1\{x\} = p_2\{\phi(x)\}$, and $p_1\big(\fringe{x}{\T_1}\big) = p_2\big(\fringe{\phi(x)}{\T_2}\big)$. \label{item:m_s:m}
  
  \item \emph{Structure preserving}. For $x,y\in\scS_1$ we have $x\in [[r_1,y]]_{\cT_1}$ if and only if $\phi(x)\in [[r_2,\phi(y)]]_{\cT_2}$.\label{item:m_s:s}
 \end{enumerate}
 
 We say that two rooted, weighted $\BR$-trees are \emph{mass-structurally equivalent} if there exists a mass-structural isomorphism from one to the other. It is straightforward to confirm that this is an equivalence relation.
\end{definition}

\begin{definition}\label{def:IP_tree}
 A rooted, weighted real tree $(\T,d,r,p)$ is an \emph{interval partition tree} (\emph{IP tree}) if it possesses the following properties.
 \begin{description}
  \item[Spanning] Every leaf of $\T$ is in the closed support of $p$, i.e.\ $\T = \Span(p)$.
  \item[Spacing] For $x\in\T$, if $x$ is either a branch point or lies in the closed support of $p$ then
   \begin{equation}\label{eq:IP_tree_spacing}
    d(r,x) + p(\fringe{x}{\T}) = 1.
   \end{equation}
 \end{description}
\end{definition}



\begin{theorem}\label{thm:MS_rep}
 Each mass-structural equivalence class of rooted, weighted $\BR$-trees contains exactly one isomorphism class of IP trees.
\end{theorem}

In light of this theorem, the isomorphism classes of IP trees can be taken as representatives of the mass-structural equivalence classes. We could refer to the isomorphism class of IP trees that are mass-structurally equivalent to a given rooted, weighted $\BR$-tree as the \emph{mass-structure} of that tree (though we will not).

\begin{definition}\label{def:hierarchy}
 A {\em hierarchy} on a finite set $S$ is a collection $\cH$ of subsets of $S$ such that 
 \begin{enumerate}[label=(\alph*), ref=(\alph*)]
  \item if $A, B \in \cH$ then $A \cap B$ equals either $A$ or $B$ or $\varnothing$, and
  \item $S \in \cH$, $\varnothing \in \cH$, and $\{s\} \in \cH$ for all $s \in S$.
 \end{enumerate}
 
 Permutations act on hierarchies by relabeling the contents of constituent sets: if $\cH$ is a hierarchy on $[n]$ and $\pi$ a permutation of $[n]$, then 
 \begin{equation*}
  \pi( \cH ):= \{ \{\pi(j): j \in A\}\colon A \in \cH\}.
 \end{equation*}
 A random hierarchy $\cH$ on $[n] := \{1,2,\ldots,n\}$ is \emph{exchangeable} if
 \begin{equation*}
  \pi(\cH) \stackrel{d}{=} \cH \qquad \text{for every permutation }\pi\text{ of }[n].
 \end{equation*}
 
 We adopt the convention that a \emph{hierarchy on $\BN$} is a sequence $(\cH_n,\,n\ge1)$, with each $\cH_n$ a hierarchy on $[n]$, with the consistency condition that
 $$\cH_n = \Restrict{\cH_{n+1}}{[n]} := \big\{ A\cap [n]\colon A\in\cH_{n+1} \big\} \qquad \text{for }n\ge1.$$
 We call $(\cH_n,\,n\ge1)$ \emph{exchangeable} if every $\cH_n$ is exchangeable. Exchangeable hierarchies on $\BN$ were studied in \cite{FormHaulPitm17}. This method of representing an infinite combinatorial object by a projectively consistent family has often been used to study exchangeable infinite structures; see \cite{MR1457625}, \cite[Chapter 2.2]{MR2245368}.
 
 A random hierarchy $(\cH_n,\,n\ge 1)$ on $\BN$ is \emph{independently generated} if for every $N$ and every vector $(A_1,\ldots ,A_k)$ of disjoint subsets of $[N]$, the restrictions $\big(\restrict{\cH_N}{A_1},\ldots,\restrict{\cH_N}{A_k}\big)$ of $\cH_N$ to these subsets are mutually independent. We write \emph{e.i.g.}\ to abbreviate ``exchangeable and independently generated.'' By way of analogy with de Finetti's theorem for exchangeable sequences, in \cite[Theorem 2]{FormHaulPitm17}, it was shown that exchangeable laws of hierarchies on $\BN$ can be represented as convex combinations of e.i.g.\ laws.
\end{definition}

A hierarchy on a finite set $S$ can be constructed by recursively partitioning the set, and then partitioning the resulting blocks, until only singletons remain. The collection of all subsets obtained at any point in this process comprise a hierarchy on $S$. Such a hierarchy can be represented as a tree rooted at $S$, with the non-empty blocks of the hierarchy being the nodes and the singleton blocks, in particular, being the leaves. 

A nested topic model is an exchangeable hierarchy used as the basis for a machine learning algorithm to arrange a collection of documents by topic and subtopic (and sub-subtopic, etc.), or to classify documents as mixtures of subtopics \cite{NestedCRP,NHDP}. Rather than being given a fixed hierarchy of topics, such algorithms infer natural topic clusterings from the set of documents they are given. Exchangeable hierarchies also relate to fragmentation and coagulation processes \cite{BertoinFragCoag}, in which sets break down or aggregate over time. Hierarchies differ from these processes in that they do not give an account of the \emph{times} at which sets join or break apart; they only describe which sets eventually arise in such a process. Hierarchies relate to other phylogenetic models, as well, such as phylogenetic trees \cite{SempleSteel}. A more complete catalog of references related to exchangeable hierarchies can be read from \cite{FormHaulPitm17}.

Now, consider a rooted, weighted $\BR$-tree $(\T,d,r,p)$. Let $(x_i,\,i\ge1)$ be an i.i.d.\ random sequence with law $p$. Set
\begin{equation}
 \cH_n := \big\{\{ i \in [n]\colon x_i\in \fringe{x}{\T}\} \colon x\in\T\big\} \qquad \text{for } n\ge1.
\end{equation}
We say that $(\cH_n,\,n\ge1)$ is \emph{derived by sampling from $(\T,d,r,p)$}. Let $\Theta(\T,d,r,p)$ denote the law of $(\cH_n,\,n\ge1)$. This is an e.i.g.\ law. If two rooted, weighted real trees are isomorphic, then $\Theta$ maps them to the same law. If $\scT$ is an isomorphism class of such trees, we write $\Theta(\scT)$ to denote the unique e.i.g.\ law that appears in the image of the class under $\Theta$.

\begin{theorem}\label{thm:MS_Theta}
 Two rooted, weighted $\BR$-trees are mass-structurally equivalent if and only if they have the same image under $\Theta$.
\end{theorem}

For a hierarchy $(\cH_n,\,n\ge1)$, we denote the associated tail $\sigma$-algebra by
\begin{equation}
 \tail(\cH_n) := \bigcap_{j\ge 1}\sigma\left(\restrict{\cH_k}{\{j,j+1,\ldots,k\}},\ k\ge j\right).
\end{equation}
We resolve \cite[Conjecture 1]{FormHaulPitm17} and strengthen Theorem 5, which was the main result of that paper, as follows.

\begin{theorem}\label{thm:hierarchies_IP}
 \begin{enumerate}[label=(\roman*), ref=(\roman*)]
  \item For $(\cH_n,\,n\ge 1)$ an exchangeable random hierarchy on $\BN$, there exists an a.s.\ unique, $\tail(\cH_n)$-measurable random isomorphism class of IP trees, $\scT$, such that $\Theta(\scT)$ is a regular conditional distribution (r.c.d.) for $(\cH_n,\,n\ge1)$ given $\tail(\cH_n)$.\label{item:hIP:rcd}
  \item The map $\Theta$ is a bijection from the set of isomorphism classes of IP trees to the set of e.i.g.\ laws of hierarchies on $\BN$.\label{item:hIP:bijection}
 \end{enumerate}
\end{theorem}

This theorem is a hierarchies analogue to Kingman's paintbox theorem \cite{Kingman78}, which describes exchangeable partitions of $\BN$, or to de Finetti's theorem for exchangeable sequences of random variables \cite{KallenbergSymm}.


We recall \cite[Example 1]{FormHaulPitm17}.

\begin{example}\label{eg:naive}
 We think of the following as a hierarchy on the interval $[0,3)$:
 \begin{equation*}
 \begin{split}
  \mathscr{H} &:= \{[0,1), [1,2), [2,3)\} \cup\left\{ \bigcup\nolimits_{n \geq 1} \left\{\left[\frac{j}{2^n}, \frac{j+1}{2^n}\right)\colon 0 \leq j \leq 2^n-1\right\} \right\}\\
 		&\qquad \cup \{[x,3)\colon 2 < x < 3\} \cup \{\{x\}\colon x\in [0,3)\} \cup \{[0,3),\varnothing\}.
 \end{split}
 \end{equation*}
 Let $(s_i,\,i\ge1)$ be an i.i.d.\ sequence of Uniform$[0,3)$ random variables, and define an e.i.g.\ hierarchy on $\BN$ by
 \begin{equation}\label{triple-H-2}
  \cH_n := \{\{i \in [n]\colon s_i \in B\}\colon B \in \mathscr{H}\} \qquad \text{for }n\geq 1.
 \end{equation}
\end{example}

\begin{figure}[thbp] 
   \centering
   \includegraphics[height=1.1in,width=4in]{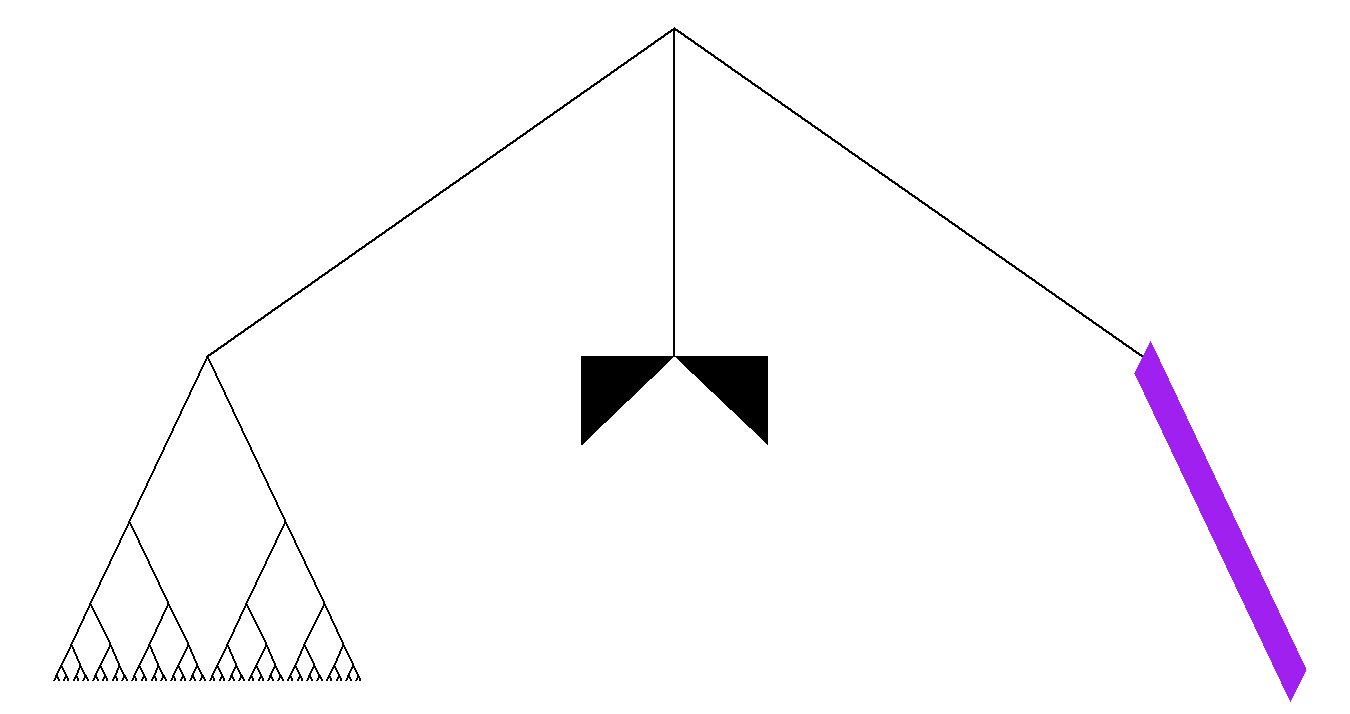}
   \caption{IP tree representation of the hierarchy in \eqref{triple-H-2}, as in Theorem \ref{thm:hierarchies_IP}\ref{item:hIP:bijection}. The wedge represents an atom. The heavy, shaded line represents continuous mass on the skeleton.\label{fig:naive}}
\end{figure}

In \cite{FormHaulPitm17}, the authors pose the ``Na\"ive conjecture'' that exchangeable hierarchies are characterized by a mixture of the three behaviors exhibited in Example \ref{eg:naive}: macroscoping branching, broom-like explosion, and comb-like erosion. This is formalized in Conjecture 2 of that paper, which is verified by Theorem \ref{thm:hierarchies_IP} above and the following.

\begin{theorem}\label{thm:measure_decomp}
 For $(\T,d,r,p)$ an IP tree, $p$ can be decomposed uniquely as $p^a + p^s + p^l$, with $p^a$ purely atomic, $p^s$ the restriction of length measure to a subset of the skeleton of $\T$, and $p^l$ a diffuse measure on the leaf set of $\T$.
\end{theorem}

The idea is that broom-like explosions in the hierarchy correspond to atoms in the measure $p^a$, comb-like erosion corresponds to diffuse measure $p^s$ on the skeleton, and macroscopic splitting corresponds to branch points, with the set of singletons that are eventually isolated by repeated splitting corresponding to continuous measure $p^l$ on the leaves.
In light of Theorems \ref{thm:hierarchies_IP} and \ref{thm:measure_decomp}, IP trees may be understood as recipes for combining and interspersing these three behaviors. Up to isomorphism, they contain no more and no less information than this.

In Section \ref{sec:real_IP_tree} we discuss a general ``bead-crushing'' construction of IP trees and the related notion of strings of beads, from \cite{PitmWink09}. Section \ref{sec:tree_construction} recounts relevant background from \cite{FormHaulPitm17} relating hierarchies to CRTs, then connects this material to IP trees. The main mathematical work of the paper is done in Section \ref{sec:key_props}, with proofs of two key propositions building towards the main results, all of which are then proved in Section \ref{sec:thm_pfs}. Finally, in Section \ref{sec:complements} we offer some final thoughts and open questions, including a discussion of the Brownian CRT in the context of the ideas of this paper.

\section{Interval partition trees}\label{sec:real_IP_tree}

We will construct IP trees as subsets of the following space.

\begin{definition}\label{def:l1}
 Let $\ell_1$ denote the Banach space of absolutely summable sequences of reals under the norm $\|(x_i,\,i\ge1)\| = \sum_i|x_i|$. We write $\ell_1(x,y) := \|y-x\|$. Let $(\Be_j,\, j \geq 1)$ be the coordinate vectors, $\Be_1 = (1, 0, 0, \ldots)$, \mbox{$\Be_2  = (0,1,0,\ldots)$}, etc.. 
 For $m \geq 1$ let $\pi_m$ denote the orthogonal projection onto span$\{\Be_1, \ldots, \Be_m\}$, and let $\pi_0$ send everything to $(0,0,\ldots)$, which we denote by $0$.
 Let \cl\ denote the topological closure map on subsets of $\ell_1$.
\end{definition}

\begin{definition}\label{def:special_path}
 Following Aldous \cite{MR1207226}, for $x\in \ell_1$ let $[[0,x]]_{\ell}$ denote the path that proceeds from 0 to $x$ along successive directions. In particular, 
 \begin{equation}\label{special-path}
  [[0,x]]_{\ell} := \{x\} \cup \bigcup_{m \geq 0} \{t\pi_m(x)+ (1-t)\pi_{m+1}(x): t\in [0,1]\}.
 \end{equation}
 For $x,y\in\ell_1$ with all non-negative coordinates, 
 $$[[0,x]]_{\ell}\cap [[0,y]]_{\ell} = [[0,z]]_{\ell}$$
 for some $z\in \ell_1$, possibly equal to zero. We define 
 \begin{equation}\label{eq:tree_wedge_def}
  (x\wedge y)_\ell := z, \qquad [[x,y]]_{\ell} := \big([[0,x]]_{\ell} \cup [[0,y]]_{\ell}\setminus [[0,z]]_{\ell}\big) \cup \{z\}.
 \end{equation}
\end{definition}

For example, if $x = 2\Be_1 + \Be_3$ then $[[0,x]]_\ell$ is a union of two segments parallel to the first and third coordinate axes, $\Be_1[0,2] \cup (2\Be_1 + \Be_3[0,1])$. Generally, if $x$ has only finitely many non-zero coordinates then the last of these segments terminates at $x$, and the singleton $\{x\}$ on the right hand side in \eqref{special-path} becomes redundant.


\begin{definition}\label{def:uniformize}
 We call a probability measure $q$ with compact support $K\subseteq [0,\infty)$ \emph{uniformized} if $q[0,x) = x$ for every $x\in K$. Let $F\colon \BR\to [0,1]$ be a cumulative distribution function for a probability measure $\mu$ on $\BR$. The \emph{uniformization} of $\mu$ is the measure $q$ on $[0,1]$ specified by $q[0,x] = \inf(\text{range}(F)\cap [x,1])$.
\end{definition}

Note that the uniformization of a measure is uniformized. 

\begin{lemma}\label{lem:1D_IPT}
 A probability measure $q$ on $\BR$ is uniformized if and only if $([0,L],d,0,q)$ is an IP tree, where $d$ is the Euclidean metric and $L$ is the maximum of the compact support of $q$.
\end{lemma}

\begin{proof}
 The Spanning property follows from our definition of $L$. The Spacing property is then equivalent to the uniformization property.
\end{proof}

\subsection{The bead-crushing construction of IP trees}\label{sec:bead_crush}

The following is an extension of the general line-breaking construction of $\BR$-trees \cite{MR1207226,CuriHaas16} (also see \cite{AldoPitm00,GoldHaas15}), modified to construct IP trees. Our construction is illustrated in Figure \ref{fig:bead_crush}. The name ``bead-crushing'' refers to strings of beads in a continuum random tree, described by Pitman and Winkel \cite{PitmWink09}, which we discuss in Section \ref{sec:string_of_beads}. We discuss Pitman and Winkel's bead-crushing construction, which differs somewhat from ours, in Section \ref{sec:ATCRT_recovery}. That construction was generalized to a larger family of self-similar CRTs by Rembart and Winkel \cite{RembWinkString}.

\begin{figure}[p]
 \centering
  $q_1$: \raisebox{-.45\height}{\includegraphics[width=1.43in,height=.54in]{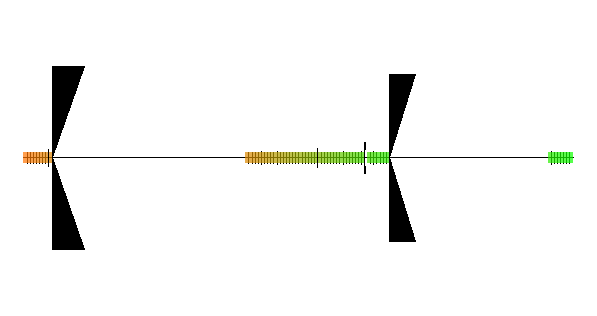}} \qquad\qquad
  $q_2$: \raisebox{-.45\height}{\includegraphics[width=1.43in,height=.54in]{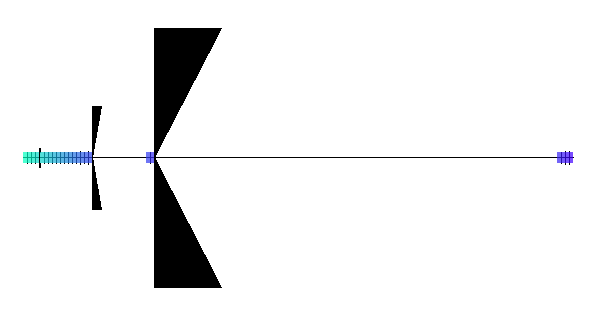}}\\
  $q_3$: \raisebox{-.45\height}{\includegraphics[width=1.43in,height=.54in]{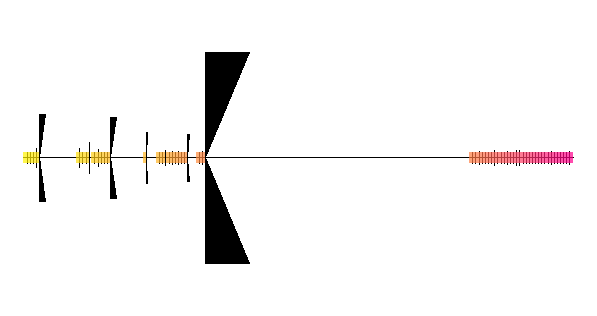}} \qquad\qquad
  $q_4$: \raisebox{-.45\height}{\includegraphics[width=1.43in,height=.54in]{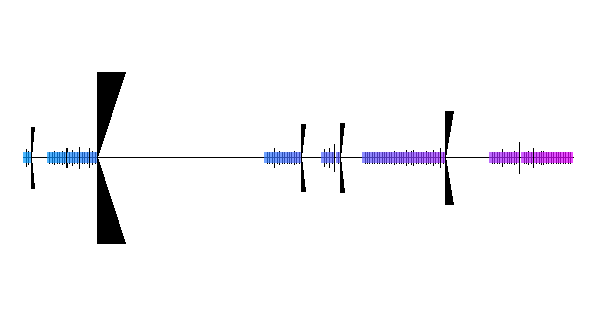}}\\[8pt]
  $\T_1$: \raisebox{-.5\height}{\includegraphics[width = 2.0in,height=1.1in]{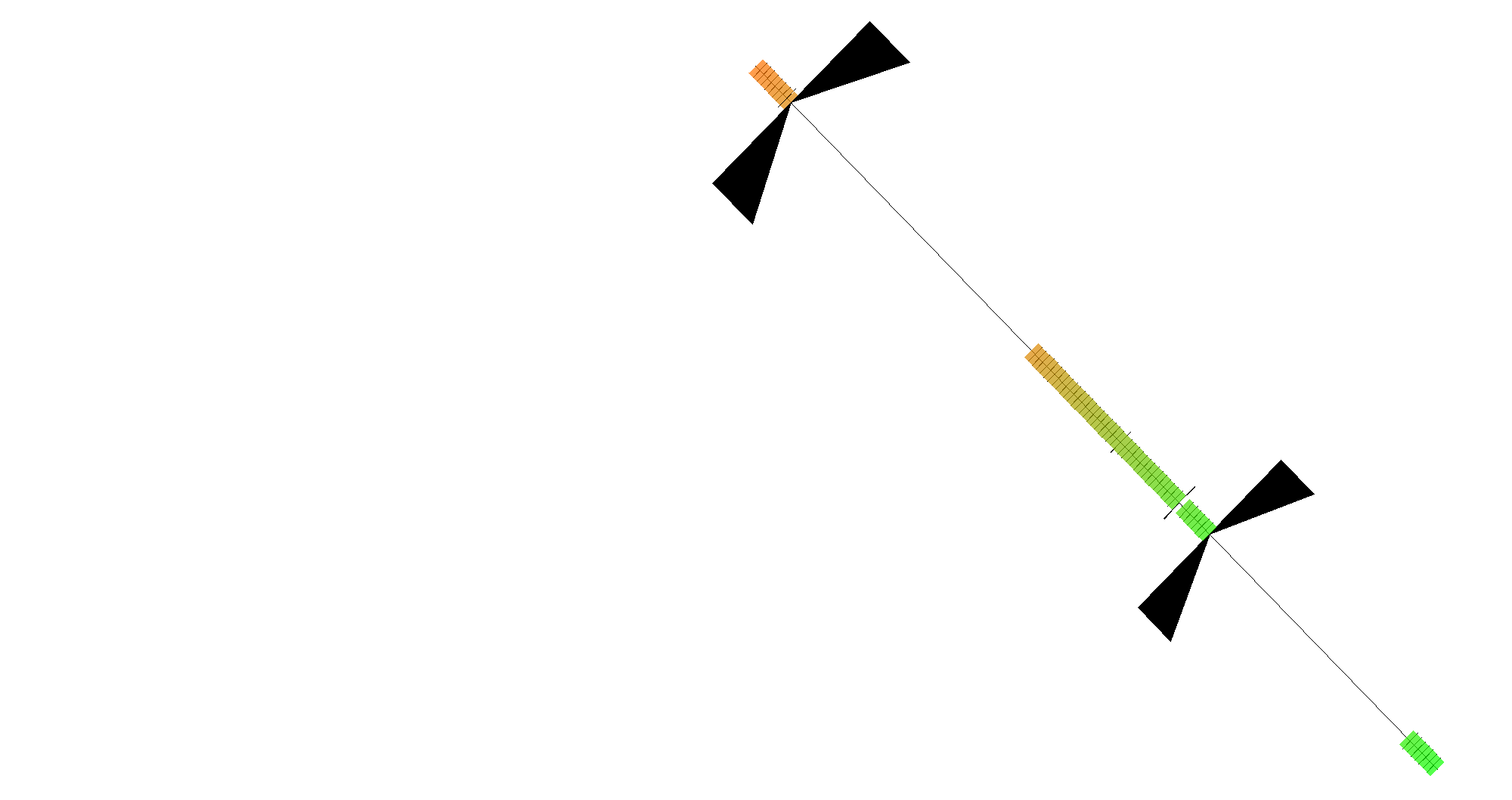}} \qquad 
  $\T_2$: \raisebox{-.5\height}{\includegraphics[width = 2.0in,height=1.1in]{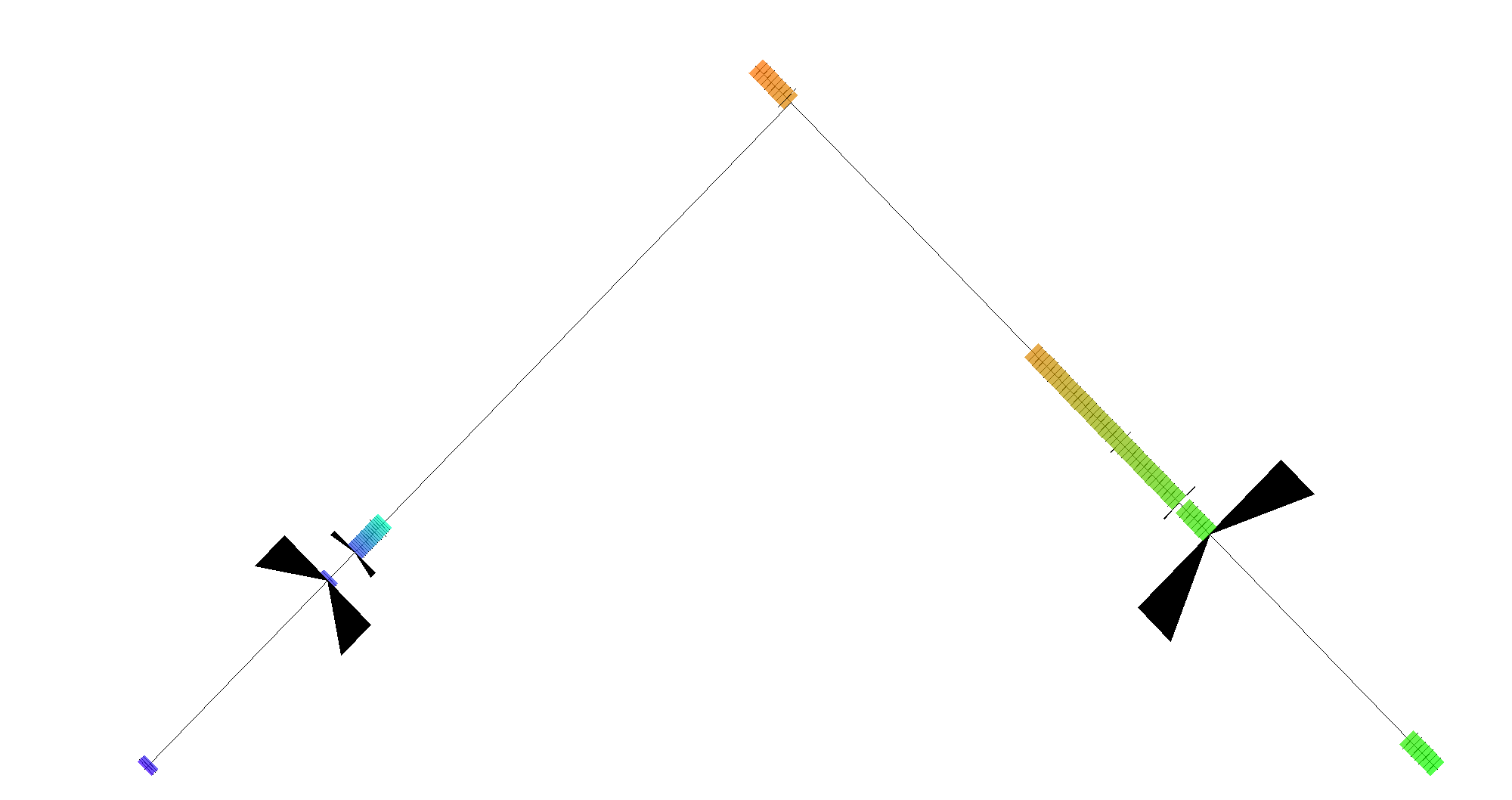}} \\[8pt]
  $\T_3$: \raisebox{-.5\height}{\includegraphics[width = 2.0in,height=1.1in]{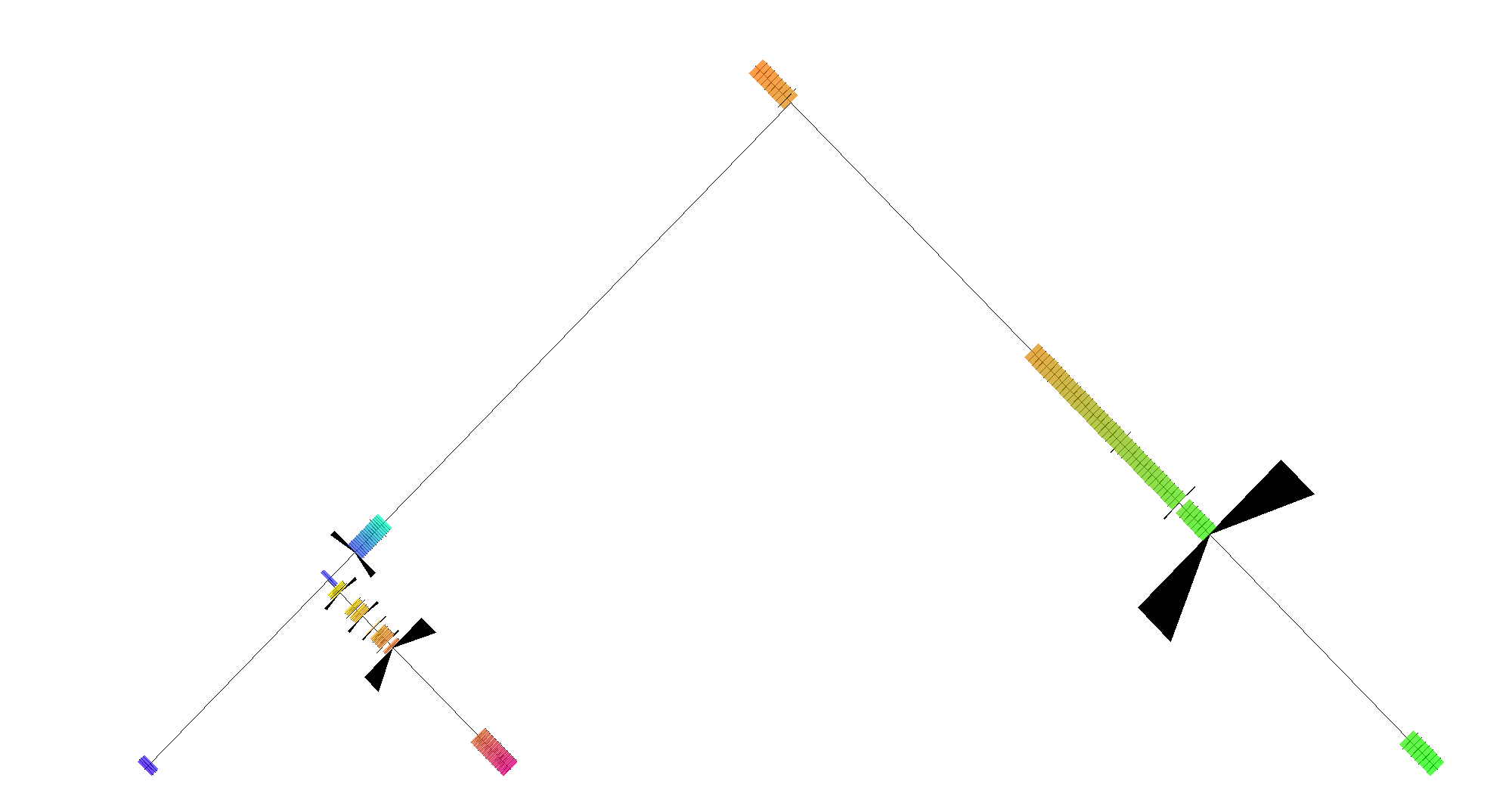}}\qquad
  $\T_4$: \raisebox{-.5\height}{\includegraphics[width = 2.0in,height=1.1in]{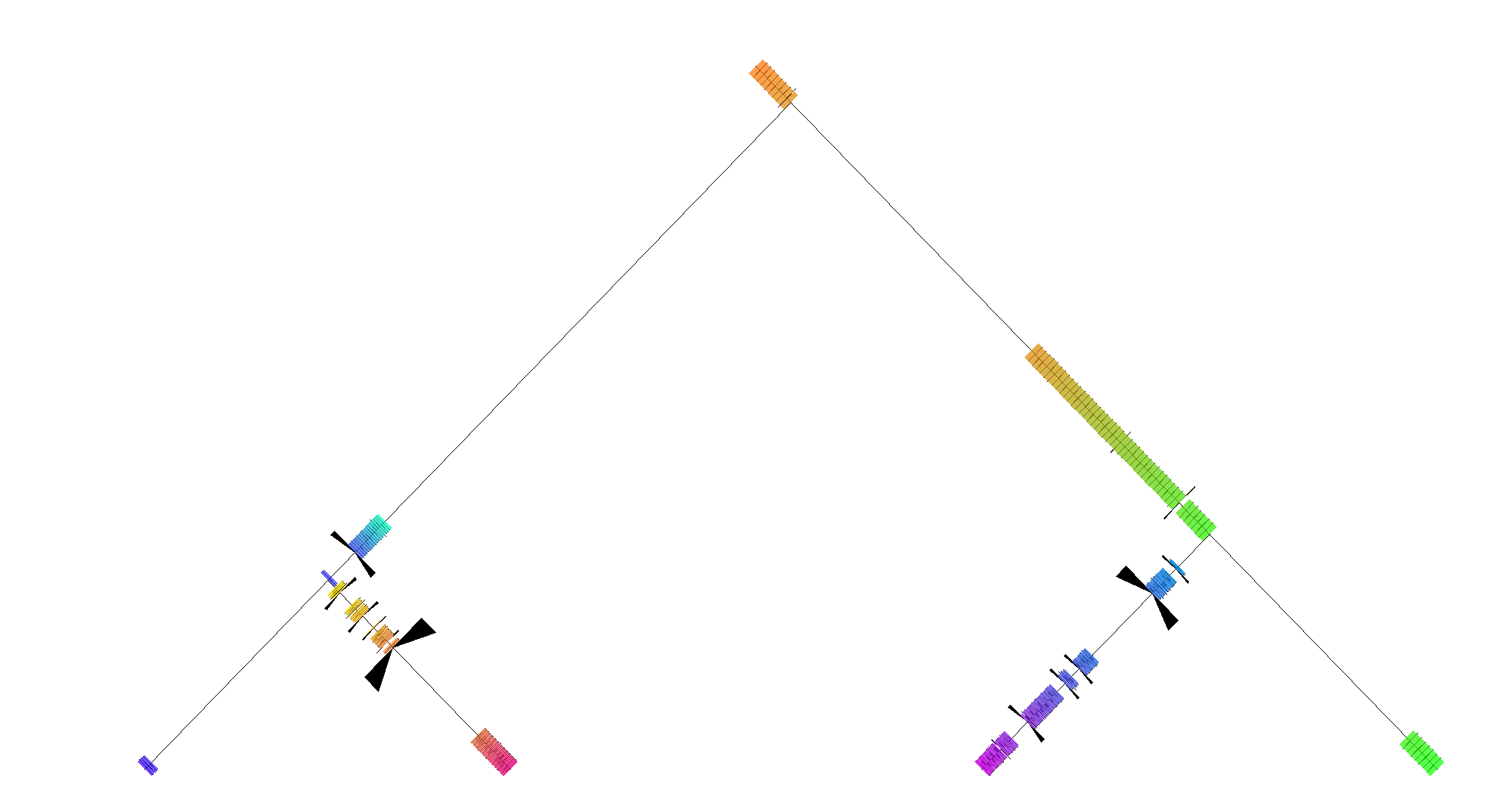}}\\[8pt]
  $\T_{120}$:\raisebox{-.5\height}{\includegraphics[width = 4.6in,height=2.53in]{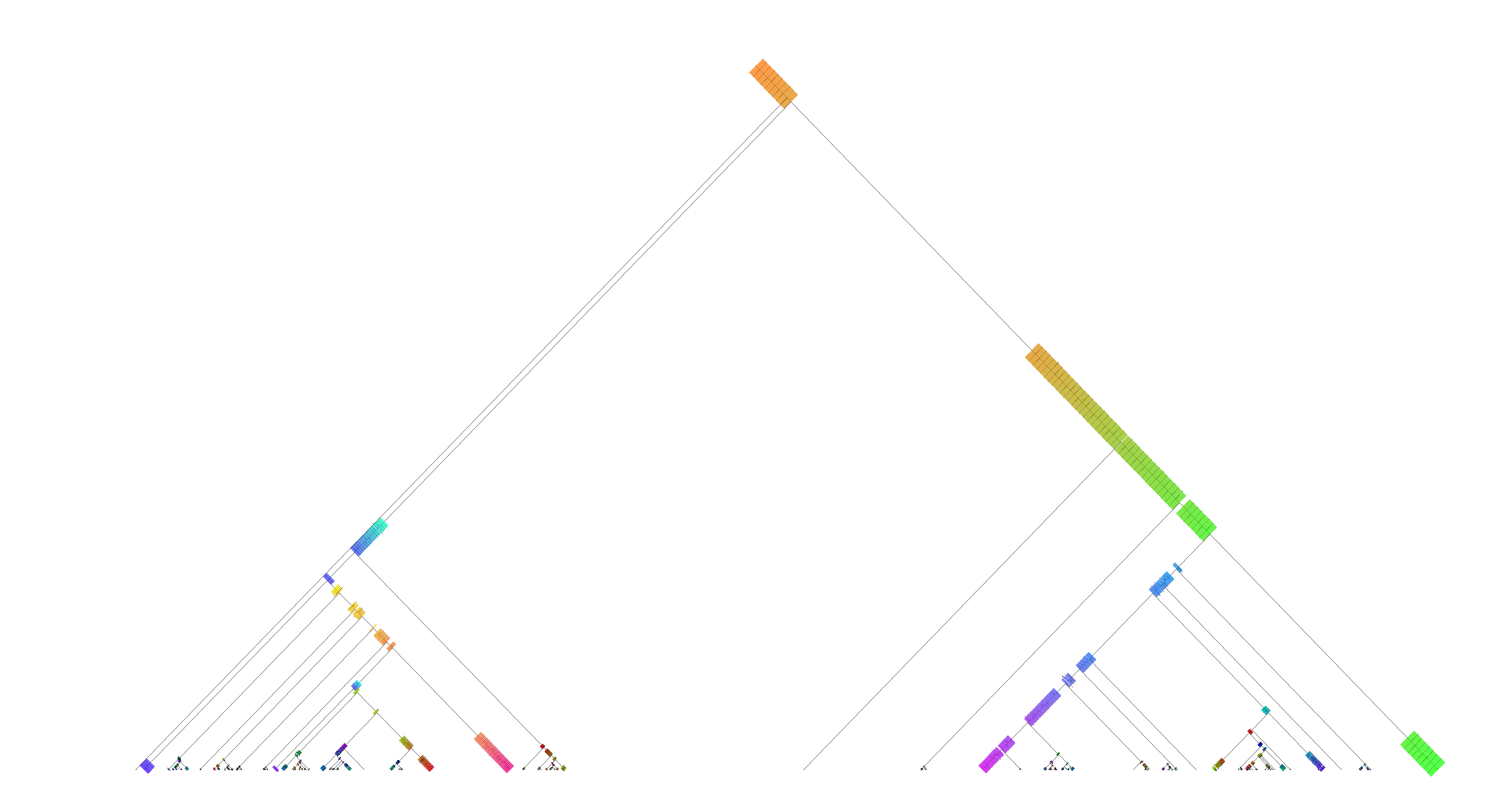}}
 \caption{The bead crushing construction described in Section \ref{sec:bead_crush}. In each tree image, the root is at the top and leaves are along a line at the bottom. Heavy, shaded lines mark subsets of the skeleton on which $q_j$ or $p_j$ equals length measure (in image of $\T_{120}$, we make these thinner to avoid branches appearing to overlap). Black wedge shapes, many of which are barely visible, represent atoms, or ``beads,'' of $q_j$ or $p_j$.\label{fig:bead_crush}}
\end{figure}

 Consider a sequence of uniformized probability measures $(q_n,\,n\ge 1)$, with $L_n := \max(\text{support}(q_n))$ for $n\ge1$. Note that each $q_n$ must have an atom $(1-L_n)\delta_{L_n}$ if $L_n < 1$. We define $\T_0 := \{0\}$ and $p_0 := \delta_0$, where here we take $0$ to denote the origin in $\ell_1$. We proceed recursively as follows.
 
 Assume $(\T_n,0,\ell_1,p_n)$ is a rooted, weighted $\BR$-tree embedded in the first $n$ coordinates in $\ell_1$. If $p_n$ has no atoms then we terminate the construction with this tree. Otherwise, fix an atom $m_n\delta_{x_n}$ of $p_n$. Fix $a_n \in (0,m_n]$. Set
 \begin{equation}\label{eq:bead_crush}
 \begin{split}
  \phi_n(z) &:= x_n + \big(p_n(\fringe{x_n}{\T_n}) + (z-1) a_n\big)\Be_{n+1} \quad \text{for }z\in [0,L_{n+1}],\\
  \T_{n+1} &:= \T_n\cup [[x_n,\phi_n(L_{n+1})]]_{\ell},\\
  p_{n+1} &:= p_n + a_n\left(-\delta_{x_n} + \phi_n\left(q_{n+1}\right)\right),
 \end{split}
 \end{equation}
 where 
 $\phi_n(q_{n+1})$ denotes the pushforward of the measure.
 
 Let $\T := \cl(\bigcup_{n \geq 1} \T_n)$. For every $N>n\ge 1$ we have $p_n = \pi_n(p_N)$, where $\pi_n$ is the projection map of Definition \ref{def:l1}. Thus, by the Daniell-Kolmogorov extension theorem, there exists a measure $p$ on $[0,1]^{\BN}$ such that $\pi_n(p) = p_n$ for every $n\ge1$. Moreover, since $(\T,\ell_1)$ is complete, $p$ is supported on $\T$.

Actually, the measures $p_n$ converge to $p$ in the first Wasserstein metric, though we will not use this.

\begin{proposition}\label{prop:bead_crush_IP}
 For any choice of sequences $(q_n,\,n\ge1)$, $(x_n,\,n\ge0)$, and $(a_n,\,n\ge0)$, the quadruple $(\T,\ell_1,0,p)$ arising from the above bead-crushing construction is an IP tree.
\end{proposition}

We first prove the following.

\begin{lemma}\label{lem:bead_crush_IP_0}
 In the setting of Proposition \ref{prop:bead_crush_IP}, the trees $(\T_n,\ell_1,0,p_n)$, $n\ge1$, that arise from bead-crushing are IP trees.
\end{lemma}

\begin{proof}
 It is easily seen that these trees possess the Spanning property, so we need only check the Spacing property. This holds by construction for $n=0$. Assume for induction that it holds from some $n\ge0$. The reader may check that for $y\in\T_n$, we get
 \begin{equation}\label{eq:bead_crushing_fringe_proj_1}
  p_{n+1}(\fringe{y}{\T_{n+1}}) = p_n(\fringe{y}{\T_n}),
 \end{equation}
 regardless of the position of $y$ relative to the point $x_n$ of insertion of the new branch. 
 Thus, $(\T_{n+1},\ell_1,0,p_{n+1})$ satisfies \eqref{eq:IP_tree_spacing} at all branch points of $\T_n$ and all points in the closed support of $p_n$. It remains to check \eqref{eq:IP_tree_spacing} at points $y\in\T_{n+1}\setminus\T_n$ in the closed support of $p_{n+1}$. By definition of $p_{n+1}$, each such $y$ equals $\phi_n(z)$ for some $z$ in the closed support of $q_{n+1}$. Thus,
 \begin{equation*}
 \begin{split}
  p_{n+1}(\fringe{y}{\T_{n+1}}) &= p_{n+1}\big(\fringe{\phi_n(z)}{\T_{n+1}}\big)
  	= a_n q_{n+1}([z,L_{n+1}]) = a_n(1 - z)\\
  	&= \big(1 - \|x_n\| - p(\fringe{x_n}{\T})\big) - a_n(z-1) = 1 - \|\phi_n(z)\|,
 \end{split}
 \end{equation*}
 where the second equality results from the definition of $p_{n+1}$, the third from the uniformized property of $q_{n+1}$ at $z$, the fourth from the Spacing property of $p_n$ at $x_n$, and the last from the definition of $\phi_n$. We conclude that $(\T_{n+1},\ell_1,0,p_{n+1})$ possesses the Spacing property, as needed for our induction.
\end{proof}

\begin{proof}[Proof of Proposition \ref{prop:bead_crush_IP}]
 %
 %
 \emph{Spacing}. By our definition of $p$ via projective consistency,
 \begin{equation}\label{eq:bead_crushing_fringe_proj}
  p_n(\fringe{y}{\T_n}) = p(\fringe{y}{\T})\qquad \text{for }n\ge0,\ y\in\T_n.
 \end{equation}
 Consider $y$ in the closed support of $p$. We will abbreviate $y_n := \pi_n(y)$. For each $n\ge1$, $y_n$ lies in the closed support of $p_n$. Therefore,
 \begin{equation}\label{eq:space_to_span}
 \begin{split}
  p\left(\fringe{y}{\T}\right) &= p\left(\bigcap\nolimits_{n\ge 1}\fringe{y_n}{\T}\right) = \lim_{n\to\infty} p(\fringe{y_n}{\T})\\
  	&= \lim_{n\to\infty}p_n(\fringe{y_n}{\T_n}) = \lim_{n\to\infty} 1-\|y_n\| = 1-\|y\|,
 \end{split}
 \end{equation}
 where the first and last equalities follow from the convergence $y_n\to y$ along the segment $[[0,y]]_{\ell}$, the second follows from the countable additivity of $p$ and the nesting $\fringe{y_n}{\T}\supseteq \fringe{y_N}{\T}$ for $n\leq N$, the third from \eqref{eq:bead_crushing_fringe_proj}, and the fourth from the Spacing property of the trees $(\T_n,p_n)$.
 
 
 \emph{Spanning}. Let $y$ be a leaf of $\T$. As before, let $y_n := \pi_n(y)$. Then for all $n\ge1$, either $y_n$ is a leaf in $\T_n$ or it lies on an atom of $p_n$, which then arises as an attachment point for a new branch later in the construction. By the Spanning property of $(\T_n,p_n)$, $y_n$ is in the closed support of $p_n$ regardless. This condition is sufficient to apply the argument \eqref{eq:space_to_span}. In particular, if $\|y\| < 1$ then $p\{y\} = p(\fringe{y}{\T}) = 1-\|y\| > 0$, so $y$ is in the closed support of $p$. 
 
 Now, suppose $\|y\| = 1$ and fix $\epsilon>0$. We will show the $\epsilon$-ball about $y$ has positive $p$-measure. Take $N$ sufficiently large so that $\|y_N\| > 1 - \epsilon/4$ and let $z$ denote the point on $[[0,y_N]]_{\ell}$ at distance $\epsilon/4$ from $y_N$. Since $\|z\| > 1-\epsilon/2$ and no point in $\T$ lies farther than one unit from the origin,
 $$\|x-y\| \leq \|x-z\| + \|y_N-z\| + \|y - y_N\| < \frac{\epsilon}{2} + \frac{\epsilon}{4} + \frac{\epsilon}{4} \quad \text{for }x\in \fringe{z}{\T}.$$
 Moreover, by the Spanning property of $\T_N$, $p(\fringe{z}{\T}) = p_N(\fringe{z}{\T_N}) > 0$. In other words, the $\epsilon$-ball about $y$ has positive measure under $p$. 
 %
\end{proof}


\begin{theorem}\label{thm:bead_crush_rep}
 Every IP tree can be isomorphically embedded in $\ell_1$ by the above bead-crushing construction.
\end{theorem}

We prove this in Section \ref{sec:thm_pfs}.

\subsection{Metrization and measurability of spaces of IP trees}

Since IP trees need not be compact, we cannot employ Hausdorff or Gromov-Hausdorff distance to metrize sets of such trees (see \cite{EvansStFleur} for discussion of such metrics). However, the only random IP trees that we will construct and consider are those arising from bead crushing. For such a tree $(\T,\ell_1,0,p)$, for every $x,y\in\T$, the segment $[[x,y]]_{\T}$ equals $[[x,y]]_{\ell}$. By the Spanning property, this means that $\T$ is specified by $p$:
\begin{equation}
 \T = \bigcup_{\text{leaves }x\in\T} [[0,x]]_{\T} = \bigcup_{x\in\text{support}(p)} [[0,x]]_{\ell}.
\end{equation}
Therefore, we can metrize the space of such trees with the Prokhorov metric on their weights:
\begin{align}
 &d_P\big( (\T,\ell_1,0,p)\; ,\; (\cS,\ell_1,0,q) \big)\\
 &\ \ = \inf\left\{ \epsilon > 0\colon \forall A\in \mathcal{B},\ p(A^\epsilon) + \epsilon \ge q(A) \text{ and } q(A^\epsilon) + \epsilon \ge p(A)  \right\},\notag
\end{align}
where $\mathcal{B}$ is the Borel $\sigma$-algebra on $\ell_1$ and $A^\epsilon$ denotes the set of all points within distance $\epsilon$ of some point in $A$. Then, we endow the space of IP trees that arise from bead crushing constructions with the resulting Borel $\sigma$-algebra. Similarly, we can metrize the space of isometry classes of IP trees with the Gromov-Prokhorov metric, which was introduced in \cite[Chapter $3\frac12{}_+$]{GromovBook} and studied in the setting of CRTs in \cite{GrevPfafWint09}. Under the Gromov-Prokhorov metric, the distance between two isometry classes of IP trees is the infimum, over all isomorphic embeddings of the two trees into a common space, of the Prokhorov distance between their embeddings. Again, we endow the space of such isometry classes with the resulting Borel $\sigma$-algebra.


\subsection{Example IP trees, strings of beads, the Brownian IP tree}\label{sec:string_of_beads}

\begin{definition}
 The \emph{simple bead-crushing construction of IP trees} is a randomization of the construction in Section \ref{sec:bead_crush} in which: (i) the measures $(q_n,\,n\ge1)$ are i.i.d.\ picks from some law on uniformized probability measures, with not all $q_n=\delta_0$; (ii) at each step, $m_n\delta_{x_n}$ is a size-biased pick from among the atoms of $p_n$; and (iii) at each step, $a_n = m_n$.
\end{definition}

This variant of the construction always yields a random IP tree with only binary branch points and a purely diffuse weight measure. Gnedin introduced uniformized measures in the context of the following bijection.

\begin{lemma}[Gnedin \cite{MR1457625}, Section 3]\label{lem:uniformization}
 The map from a uniformized probability measure to the relative complement of its support in $[0,1)$ is a bijection onto the set of open subsets of $(0,1)$. Its inverse can be described as follows. Consider $U = \bigcup_i(a_i,b_i)$, where this is a disjoint union. Then $U$ is the relative complement in $[0,1)$ of the support of the uniformized probability measure $q = q^a+q^d$, where $q^a = \sum_i (b_i-a_i)\delta_{a_i}$ and $q^d$ is the restriction of Lebesgue measure to $[0,1)\setminus U$.
\end{lemma}

In light of this lemma and the bead-crushing construction, we can construct interesting IP trees by looking at interesting open sets.

\begin{example}[Fat Cantor IP trees]\label{eg:fat_Cantor}
 Let $A_0 := [0,1]$. Let $A_1 := A_0\setminus (3/8,5/8)$. We carry on recursively, as follows. For $n\geq 1$, $A_n$ comprises $2^n$ disjoint closed intervals of the same length. We form $A_{n+1}$ by removing an open interval of length $4^{-n-1}$ from the middle of each component of $A_n$. This sequence decreases to a \emph{fat Cantor set} $A_{\infty} = \bigcap_{n\ge1}A_n$, also called a Smith-Volterra-Cantor set, with Lebesgue measure $1/2$; see \cite[p.\ 89]{MR1996162}.
 
 The fat Cantor set is closed. Let $q$ denote the unique uniformized probability measure supported on $A_{\infty}$. This equals the restriction of Lebesgue measure to $A_{\infty}$, plus a sum of atoms at the left end of each interval removed in the construction, with mass equal to the length of the removed interval. By Lemma \ref{lem:1D_IPT}, $([0,1],d,0,q)$ is an IP tree, where $d$ is Euclidean distance. If we carry out the simple bead-crushing construction with a sequence of copies $q_n = q$, then we get a binary branching IP tree with length measure interspersed among the branch points in such a way that the support of the measure does not include any non-trivial segments. See Figure \ref{fig:IPT_egs}.
\end{example}



Let $(\T,d,r,p)$ be a rooted, weighted real tree, and fix $x\in\T$. Consider the decomposition of $\T$ into the path $[[r,x]]$, called a \emph{spine}, and the collection of subtrees, called \emph{bushes}, branching out from the branch points along the spine, with perhaps a final bush rooted at $x$, if $x$ is not a leaf. This decomposition has been studied in \cite{MR1641670,MR2546748,PitmWink09}. The bushes are totally ordered by increasing distance from the root. We may project $p$ down onto the spine, replacing the mass distribution over each bush with an atom at the root of the bush. The resulting measure is called a \emph{string of beads}, with the spine being the string and the atoms of the projection of $p$ comprising the beads; see Figure \ref{fig:string_discr}. This approach was introduced in \cite{PitmWink09}.

\begin{figure}
 \centering
 \includegraphics[scale=.65]{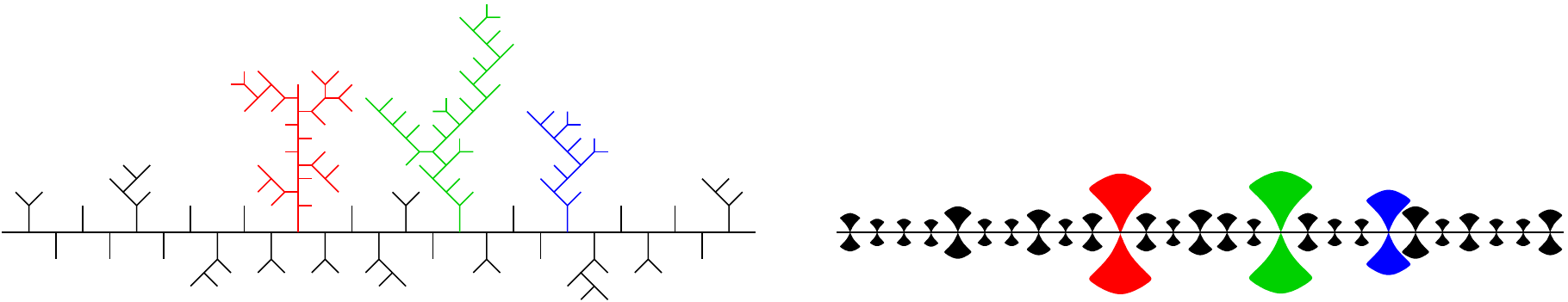}
 \caption{A string of beads in a discrete tree.\label{fig:string_discr}}
\end{figure}

\begin{example}\label{eg:aa_string}
 The \emph{two-parameter Poisson-Dirichlet distributions} \cite{PitmYor97}, denoted by $\texttt{PoiDir}(\alpha,\theta)$ with $\alpha\in [0,1)$ and $\theta>-\alpha$, are probability distributions on the Kingman simplex: the set of non-increasing sequences of real numbers that sum to 1. These distributions, introduced in \cite{Kingman75,PermPitmYor92,PitmYor97}, arise in many mathematical settings and applications. 
 Fix $\alpha\in (0,1)$. Let $(U_i,\,i\ge1)$ be i.i.d.\ \texttt{Uniform}$[0,1]$, and let $(P_i,\,i\ge 1)$ be independent of this sequence with $\texttt{PoiDir}(\alpha,\alpha)$ distribution. We define
 \begin{equation}\label{eq:aa_string}
  L := \lim_{n\to\infty}n(P_n)^{\alpha}\Gamma(1-\alpha) \quad \text{and} \quad \mu := \sum_{i\ge 1} P_i\delta_{U_iL}.
 \end{equation}
 The quantity $L$, called the \emph{$\alpha$-diversity} or sometimes the \emph{local time}, is known to be a.s.\ positive and finite, with a known probability distribution; see \cite[eqn.\ 83]{Pitman03} or \cite[eqn.\ 6]{PitmWink09}. The measure $\mu$ is called an \emph{$(\alpha,\alpha)$-string of beads}.
\end{example}
 
 In Section \ref{sec:ATCRT_recovery}, we describe the bead crushing construction of \cite{PitmWink09}, which differs from that in Section \ref{sec:bead_crush}. In particular, plugging i.i.d.\ $\big(\frac12,\frac12\big)$-strings of beads into the former construction yields a Brownian CRT.
 

\begin{definition}\label{def:AT_IP_tree}
 Fix $\alpha\in (0,1)$. Let $(q_n,\,n\ge 1)$ be a sequence of i.i.d.\ random probability measures on $[0,1]$, with each distributed as the uniformization of an $\big(\alpha,\alpha\big)$-string of beads. Let $(\T,\ell_1,0,p)$ denote the IP tree resulting from a bead-crushing construction from this sequence, as in Section \ref{sec:bead_crush}, with each $x_n$ being the location of a size-biased random atom of $p_n$ and each $a_n = m_n$. We call the resulting IP tree an \emph{$(\alpha,\alpha)$-IP tree}. In the case $\alpha=\frac12$, we call it a \emph{Brownian IP tree}. See Figure \ref{fig:IPT_egs}.
\end{definition}

This construction can be carried out with the full two-parameter family of $(\alpha,\theta)$-strings, with $\theta\ge 0$, introduced in \cite{PitmWink09}. We discuss the connection between the Brownian CRT and the Brownian IP tree in Section \ref{sec:ATCRT_recovery}.

\begin{figure}
 \centering
  (a) \raisebox{-.5\height}{\includegraphics[width = 2.0in,height=1.1in]{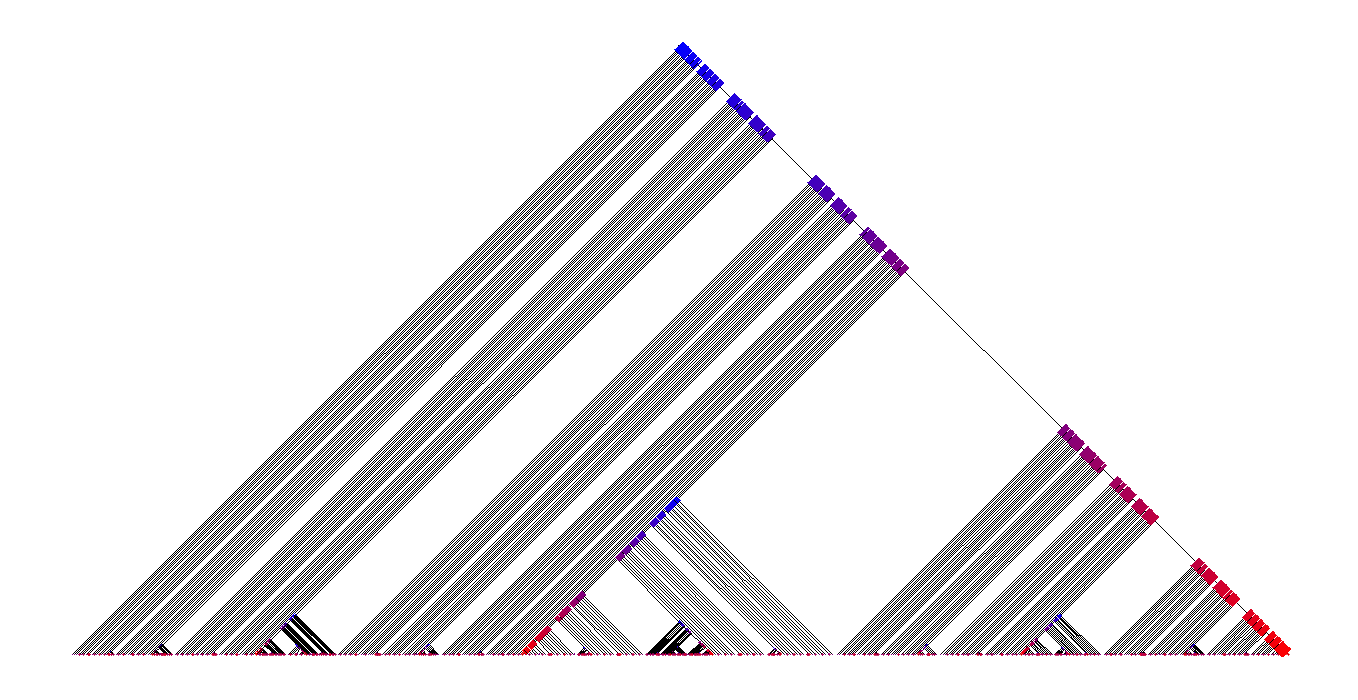}} \quad 
  (b) \raisebox{-.5\height}{\includegraphics[width = 2.0in,height=1.1in]{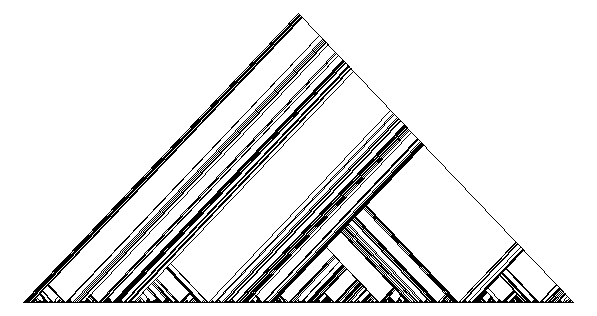}}\\
  (c) \raisebox{-.5\height}{\includegraphics[width = 2.0in,height=1.1in]{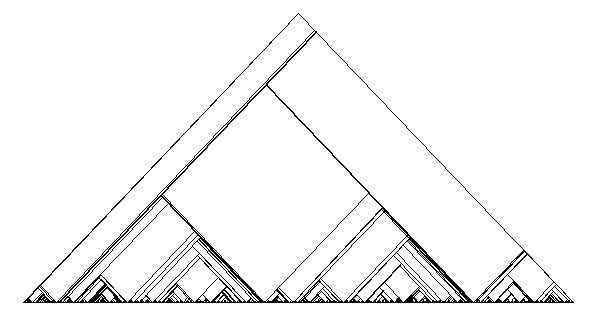}} \quad 
  (d) \raisebox{-.5\height}{\includegraphics[width = 2.0in,height=1.1in]{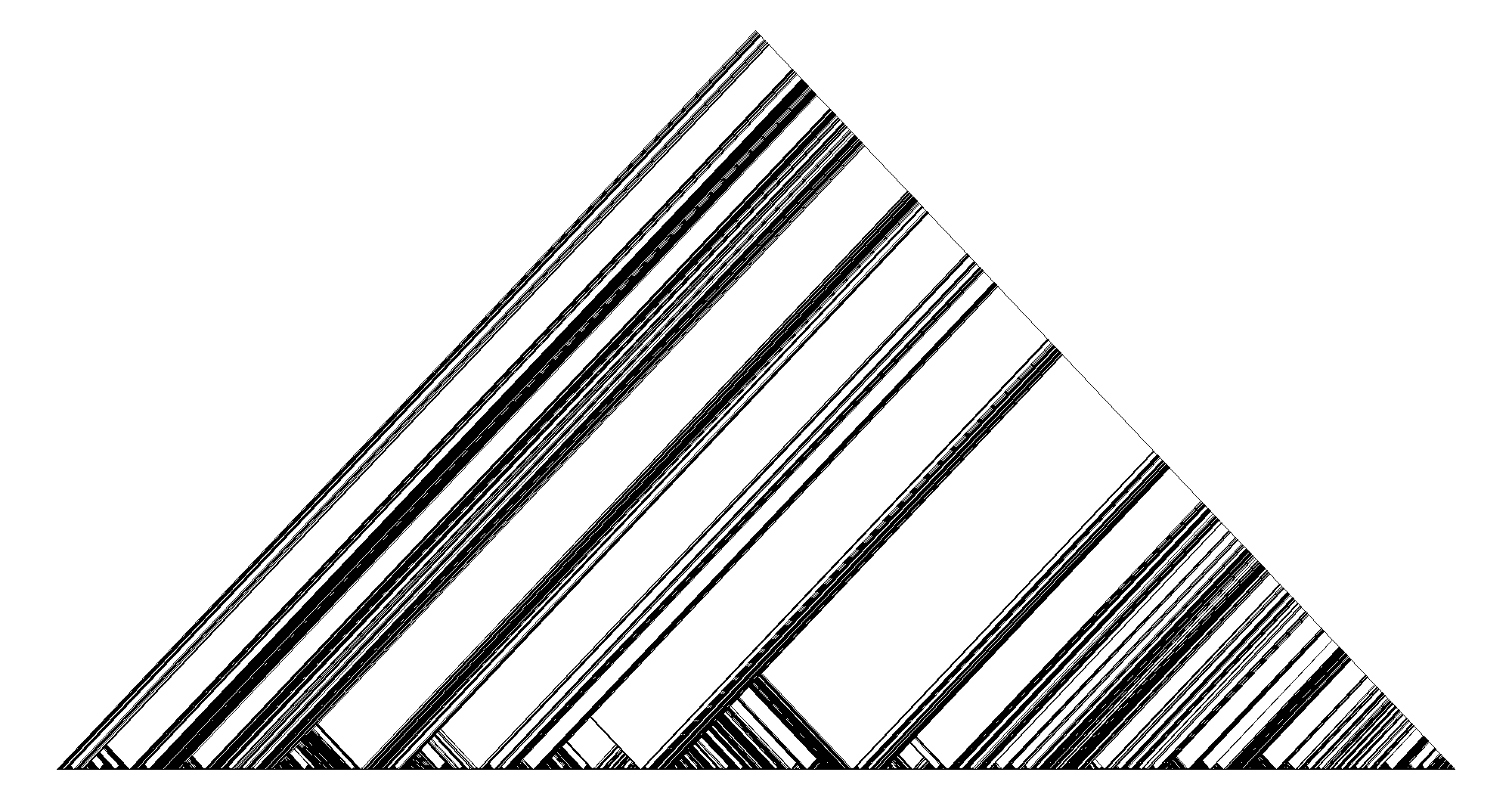}}
 \caption{Simulated IP trees, represented as in Figure \ref{fig:bead_crush}. It looks as though leaves from different branches are touching; this is not intended. (a) Fat Cantor IP tree. (b) Brownian IP tree. (c) $(.2,.2)$-IP tree. 
 (d) $(.8,5)$-IP tree.\label{fig:IPT_egs}}
\end{figure}

\section{IP tree representation of an exchangeable hierarchy}\label{sec:tree_construction}

We recall some definitions and results from \cite{FormHaulPitm17}.

\begin{definition}\label{def:MRCA}
 If $\cH$ is a hierarchy on a finite set $S$, then for $x,y\in S$, the \emph{most recent common ancestor (MRCA)} of $x$ and $y$ is
 \begin{equation} \label{eq:MRCA_def}
  (x \wedge y) := \bigcap_{G\in \cH\colon x,y \in G} G.
 \end{equation}
 If $(\cH_n,\,n\ge1)$ is hierarchy on $\mathbb{N}$, then we define the MRCA of $i$ and $j$ in this hierarchy to be
 \begin{equation}
  ( i \wedge j) := \bigcup_{n \geq \max\{i,j\}}(i \wedge j)_n,
 \end{equation}
 where $(i \wedge j)_n$ denotes the MRCA of $i$ and $j$ in $\cH_n$.
\end{definition}

MRCAs in hierarchies on $\BN$ are projectively consistent\,\cite[Proposition 1]{FormHaulPitm17}:
\begin{equation}\label{eq:MRCA_projection}
 (i\wedge j)_n = (i\wedge j)_N\cap [n] = (i\wedge j)\cap [n]\qquad \text{for }i,j\le n\le N.
\end{equation}

When constructing a tree representation of a hierarchy, we find it convenient to work with a hierarchy on $\BZ$. Let $(\cH'_n, n \geq 1)$ be an exchangeable hierarchy on $\mathbb{N}$ and let $b: \BN\to \BZ$ denote the bijection that sends odd numbers to sequential non-positive numbers and evens to sequential positive numbers. For $n \geq 1$ set 
\begin{equation}\label{eq:hier_on_Z}
 \cH_n := \left\{\left\{b(k): k \in A\right\}\colon A \in \cH'_{2n+1} \right\}.
\end{equation}
Then $\cH_n$ is a hierarchy on $[\pm n] := \{-n, \ldots, 0, \ldots, n\}$ and $\Restrict{\cH_{n+1}}{[\pm n]}  = \cH_{n}$ for every $n \geq 1$. Definition \ref{def:MRCA} extends to this context without modification.

\begin{proposition}[\cite{MR1457625} Theorem 11, \cite{MR1104078} Theorem 5, \cite{FormHaulPitm17} Proposition 2]\label{prop:spinal}
 Let $(\cH_n,\,n\ge 1)$ be an exchangeable hierarchy on $\BZ$.
 \begin{enumerate}[label=(\roman*), ref=(\roman*)]
  \item For $i,j\in \BZ$, the following limit exists almost surely:\label{item:spinal:def}
  \begin{equation}\label{eq:spinal_def}
   X^i_j := 1 - \lim_{n \to \infty} \frac{\#((i\wedge j)\cap[\pm n])}{2n}.
  \end{equation}
  \item 
   For bijections $\sigma\colon\BZ\to\BZ$ with finitely many non-fixed points,\label{item:spinal:exch}
   \begin{equation}\label{eq:spinal_exch}
    \left(X^i_j;\ i,j\in\BZ,\,i\neq j\right) \stackrel{d}{=} \left(X^{\sigma(i)}_{\sigma(j)};\ i,j\in\BZ,\,i\neq j\right).
   \end{equation}
   In particular, for $i\in\BZ$, the family $(X^i_j,\,j\in\BZ\setminus\{i\})$ is exchangeable.
  \item For $i,j,k\in\BN$, the following events are almost surely equal:\label{item:spinal:equiv}
   \begin{equation}\label{eq:spinal_equiv}
    \{X^{i}_j \leq X^{i}_k\} = \{(i \wedge k) \subseteq (i \wedge j)\} = \{k \in (i\wedge j)\}.
   \end{equation}
  \end{enumerate}
\end{proposition}


Recall the notation of Definitions \ref{def:l1} and \ref{def:special_path} for a standard basis $(\Be_n,\,n\ge1)$, projection maps $(\pi_n,\,n\ge1)$, and segments $[[0,x]]_{\ell}$ in $\ell_1$. We adopt the convention that for $k<0$, $[k] := \{k,k+1,\ldots,-1\}$.

\begin{definition}\label{def:constr_samples}
 For all $j \in \BZ$, set $t^{0}_j = 0$ and for every $k\le 0$,
 \begin{equation}\label{t-defn}
 \begin{split}
  t^{k-1}_j &:= t^k_j + \Be_{|k-1|}\left(X^{k-1}_j - \left\|t^k_j\right\|\right)_+\quad \text{for }j\in\BZ\setminus [k-1],\\
  \T_k &:= \cl\left(\bigcup\nolimits_{j\geq 1} \left[\!\left[0, t^k_j\!\right]\right]_{\ell} \right),
 \end{split}
 \end{equation}
 where $(a)_+ := \max\{a,0\}$. We treat $0$ as the root of each of the trees.
\end{definition}

Definition \ref{def:constr_samples} can be described as follows: to define the samples $(t^{k-1}_j,\,j\in \BZ\setminus [k-1])$ for some $k\leq -1$, we select a subset of the $(t^k_j,\,\in \BZ\setminus [k-1])$, possibly empty, and push these out in the $\Be_{|k-1|}$-direction, orthogonal to $\T_k$. For example, trivially, $\pi_{|i|}(t^k_j) = t^i_j$ for all $k<i<0$.


\begin{proposition}[Lemma 1 and Propositions 4, 5, 6 of \cite{FormHaulPitm17}]\label{prop:tree_limits}
 \begin{enumerate}[label=(\roman*),ref=(\roman*)]
  \item \emph{Line-breaking property of $\T$}.  For $k\leq -1$ and $j\in\BZ\setminus [k-1]$, if $t^{k-1}_j\neq t^k_j$ then $t^k_j = t^k_{k-1}$. Informally, all samples that are ``pushed out'' in passing from $t^k_j$ to $t^{k-1}_j$ are selected from the same spot on $\T_k$, namely $t^k_{k-1}$. 
   Moreover, regardless of whether $t^{k-1}_j=t^k_j$,\label{item:line_breaking}
  \begin{equation}
   \left(X^{k-1}_j - \left\|t^k_j\right\|\right)_+ = \left(X^{k-1}_j - \left\|t^k_{k-1}\right\|\right)_+.\label{eq:new_branch_X}
  \end{equation}
  
  \item  For each $j\ge 1$, the sequence $(t^k_j,\,k<0)$ converges a.s.\ in $\ell_1$. Call the limit $t_j$. \label{item:tree_limits}
  We define
  \begin{equation*}
   \T := \cl\left(\bigcup\nolimits_{k<0} \T_k\right).
  \end{equation*}
  
  \item The family $(t_j,\,j\geq 1)$ is exchangeable and has a driving measure $p$. Likewise, for every $k<0$, the family $(t^k_j, j\geq 1)$ is exchangeable and has a driving measure $p_k$.\label{item:driving}
  
  \item  For distinct $u,v\in\BN$,\label{item:MRCA_equals_fringe}
  \begin{equation}
   (u\wedge v)_{\cH}\cap\BN = \{j\in\BN\colon t_j\in \fringe{(t_u\wedge t_v)_{\ell}}{\T}\}.
   \label{eq:MRCA_equals_fringe}
  \end{equation}
 \end{enumerate}
\end{proposition}

\begin{theorem}[Theorem 5 and its proof in \cite{FormHaulPitm17}]\label{thm:hier_rep_constr}
 The random law $\Theta(\T,\ell_1,0,p)$ is a r.c.d.\ for $(\cH'_n,\,n\ge1)$ on $\tail(\cH'_n)$.
\end{theorem}

To this description, we add the following.

\begin{proposition}\label{prop:its_a_bead_crush}
 The quadruple $(\T,\ell_1,0,p)$ is a random IP tree arising from a bead crushing construction as in Section \ref{sec:bead_crush}, with the caveat that at some steps $k$, $(\T_{k-1},p_{k-1}) = (\T_{k},p_{k})$.
\end{proposition}

The description of bead crushing in Section \ref{sec:bead_crush} does not always allow this possibility of the tree going unchanged in one of the steps. We refer to this variant of bead crushing as \emph{bead crushing with pauses}. Of course, trees arising from the construction with pauses are still IP trees.

\begin{proof}
 For convenience, we restate \eqref{eq:bead_crush} for use in the present setting:
 \begin{equation}\label{eq:iabc:bead_crush}
 \begin{split} 
  \phi_k(z) &:= t^k_{k-1} + \big(p_k(\fringe{t^k_{k-1}}{\T_k}) + (z-1) a_k\big)\Be_{n+1} \quad \text{for }z\in [0,L_{k-1}],\\
  \T_{k-1} &:= \T_k\cup [[t^k_{k-1},\phi_k(L_{k-1})]]_{\ell},\\
  p_{k-1} &:= p_k + a_k\left(-\delta_{t^k_{k-1}} + \phi_k\left(q_{k-1}\right)\right),
 \end{split}
 \end{equation}
 where $L_{k-1} = \max(\textnormal{support}(q_{k-1}))$. 
 We will prove that, at each step in the iterative construction of $(\T,p)$, if $(\T_{k-1},p_{k-1})\neq (\T_{k},p_k)$ then there exists a uniformized law $q_{k-1}$ and a mass $a_k\in (0,m_k]$, where $m_k := p_k\{t^k_{k-1}\} > 0$, such that $(\T_{k-1},p_{k-1})$ is obtained from $(\T_{k},p_k)$ as in \eqref{eq:iabc:bead_crush}.
 
 Base step: $k=0$. By definition, $t^0_{-1} = 0$ and $t^{-1}_j = X^{-1}_j\Be_1$ for each $j\neq -1$. We set $a_0 := m_0 = 1$. Then, following \eqref{eq:iabc:bead_crush}, $\phi_{-1}(z) = z\Be_1$ for $z\in [0,1]$. Let $q_{-1}$ denote the driving measure of the sequence $(X^{-1}_j,\,j\ge 1)$. 
 \begin{equation}\label{eq:iabc:spacing}
 \begin{split}
  1 - X^{-1}_j &= \lim_{n\to\infty}\frac{\# ((-1\wedge j)\cap [\pm n])}{2n}\\
  	&= \lim_{n\to\infty}\frac{\# \{i\in [\pm n]\colon X^{-1}_i \ge X^{-1}_j\}}{2n} = q_{-1}[X^{-1}_j,1],
 \end{split}
 \end{equation}
 where the first equation follows from \eqref{eq:spinal_def}, the second from \eqref{eq:spinal_equiv}, and the last from the definition of $q_{-1}$. Since the $X^{-1}_j$ are dense in the closed support of $q_{-1}$, we find that $q_{-1}$ is uniformized, in the sense of Definition \ref{def:uniformize}. Since $t^{-1}_j = \phi_{-1}(X^{-1}_j)$ and $p_{-1}$ is the driving measure of the $(t^{-1}_j)$, we conclude that $p_{-1} = \phi_{-1}(q_{-1})$, consistent with the last line of \eqref{eq:iabc:bead_crush}. Thus $(\T_1,\ell_1,0,p_1)$ is an IP tree arising from a single step of a bead crushing construction.
 
 Inductive step. Fix $k < 0$ and assume that $(\T_k,\ell_1,0,p_k)$ is an IP tree arising from $|k|$ steps of the bead crushing construction with pauses. Let
 \begin{equation*}
  S := \big\{j\in\BZ\setminus [k-1]\colon ((k-1) \wedge j)_{\cH} \cap [k] = \emptyset\big\}.
 \end{equation*}
 Informally, $S$ is the set of indices that remain in a block with $k-1$ in the hierarchy until after $k-1$ has branched away from all of the indices $k,k+1,\ldots,-1$. By \eqref{eq:spinal_equiv} and the definition of the $(t^i_j)$,
 \begin{equation*}
  S = \big\{j\in\BZ\setminus [k-1]\colon X^{k-1}_j > \max_{i\in [k]} X^i_j\big\}
  		= \big\{j\in\BZ\setminus [k-1]\colon t^{k-1}_j\neq t^k_j\big\}.
 \end{equation*}
 Thus, $S = \emptyset$ if and only if $(\T_{k-1},p_{k-1}) = (\T_k,p_k)$, in which case we have nothing to prove. So assume $S\neq\emptyset$. 
 
 The family $(\cf\{j\in S\},\,j\in \BZ\setminus [k-1])$ is exchangeable, and $S\neq\emptyset$ means that not all entries are zero, so by de Finetti's theorem,
 $$a_k := \lim_{n\to\infty}\frac{\#(S\cap [\pm n])}{2n} > 0.$$
 By Proposition \ref{prop:tree_limits}\ref{item:line_breaking}, every index $j\in S$ satisfies $t^k_j = t^k_{k-1}$. Thus, $a_k$ is bounded above by $m_k = p_k\{t^k_{k-1}\}$.
 
 Now, for $j\in S$, let
 \begin{equation}\label{eq:new_branch_spinal}
  Y_j := 1 - \frac{1-X^{k-1}_j}{a_k} = 1-\lim_{n\to\infty}\frac{\# \big(((k-1) \wedge j)\cap [\pm n]\big)}{\# (S\cap[\pm n])}.
 \end{equation}
 Here, the rightmost formula follows by plugging in the definitions of $a_k$ and $X^{k-1}_j$ and canceling out factors of $2n$. Let $f$ denote the unique increasing bijection from $\BN$ to $S\cap\BN$. The sequence $(Y_{f(j)},\,j\in \BN)$ is exchangeable; let $q_{k-1}$ denote its driving measure. 
 By an argument similar to that in \eqref{eq:iabc:spacing}, $Y_j = 1 - q_{k-1}[Y_j,1]$ for each $j\in S$. Since the $(Y_j,\,j\in S)$ are dense in the closed support of $q_{k-1}$, we conclude that $q_{k-1}$ is uniformized. 
 
 Now, consider the map $\phi_k$ as defined in \eqref{eq:iabc:bead_crush}. Note that
 \begin{equation*}
  \big\|\phi_k(Y_j)\big\| = \big\|t^k_{k-1}\big\| + p_k(\fringe{t^k_{k-1}}{\T}) + a_k\big(Y_j -1\big)
  		= X^{k-1}_j,
 \end{equation*}
 where the second equality follows by appealing to the Spacing property of $(\T_k,p_k)$ at $t^k_{k-1}$ and plugging in the definition of $Y_j$. Thus, for $j\in S$, $\phi_k(Y_j)$ is a point embedded in the first $|k|+1$ coordinates in $\ell_1$ whose projection onto the first $|k|$ coordinates is $t^k_{k-1}$, and with $|k|+1^{\text{st}}$ coordinate equal to $X^{k-1}_j-\big\|t^k_{k-1}\big\|$. We conclude that $\phi_k(Y_j) = t^{k-1}_j$. Since $p_{k-1}$ is the driving measure for the sequence $(t^{k-1}_j,\,j\ge1)$, we find that it satisfies the third formula in \eqref{eq:iabc:bead_crush}. Therefore, $(\T_{k-1},\ell_1,0,p_{k-1})$ is an IP tree arising from $|k|+1$ steps of a bead-crushing construction with pauses, which completes our induction.
\end{proof}

\section{Two key propositions}\label{sec:key_props}

To prove our theorems we require two more major intermediate steps. 
%
Let $(\cS,d,r,q)$ be a rooted, weighted $\BR$-tree, let $(s_i,\,i\in\BZ)$ denote i.i.d.\ samples from $q$, and let $(\cH_n,\,n\ge 1)$ denote the hierarchy on $\BZ$ derived from $(\cS,d,r)$ via these samples. In other words, modulo our choice to label with $\BZ$ rather than $\BN$, $(\cH_n)$ is exchangeable and independently generated (e.i.g.) with law $\Theta(\cS,d,r,q)$. Let $(\T,\ell_1,0,p)$ and $(t_j,\,j\geq 1)$ denote the random IP tree and samples that arise from applying the construction of Section \ref{sec:tree_construction} to $(\cH_n)$. 


\begin{proposition}\label{prop:MSE_to_constr_tree}
 For every rooted, weighted $\BR$-tree, there is a deterministic bead-crushing construction, as in Section \ref{sec:bead_crush}, that yields an IP tree that: (i) is mass-structurally equivalent to $(\cS,d,r,q)$ and (ii) has the same image under $\Theta$ as $(\cS,d,r,q)$. In particular, the law of the random IP tree $(\T,\ell_1,0,p)$ is supported on the set of such trees.
\end{proposition}

We prove this proposition in Section \ref{sec:MSE_to_constr_tree}. Then, in Section \ref{sec:IP_tree_MSE} we prove the following.

\begin{proposition}\label{prop:IP_tree_MSE}
 If two IP trees are mass-structurally equivalent then they are isomorphic.
\end{proposition}


To prove these propositions we require two lemmas.  Extending the notation $(y\wedge z)_{\ell}$ of Definition \ref{def:special_path}, for $y,z\in\cS$, let $(y\wedge z)_{\cS}$ denote the unique point in the intersection $[[r,y]]\cap[[r,z]]\cap [[y,z]]$. This equals the branch point that separates $y$, $z$, and $r$, except in the degenerate circumstance that all three lie on a common segment, in which case $(y\wedge z)_{\cS}$ equals whichever of $y$, $z$, or $r$ lies between the other two.

\begin{lemma}\label{lem:span_samp_approx}
 It is a.s.\ the case that for every $j\in\BN$ and $\epsilon>0$, there is some $i\neq j$ for which $d((s_i\wedge s_j)_{\cS},s_j) < \epsilon$.\end{lemma}

\begin{proof}
 Fix $j\in\BN$ and $\epsilon>0$. Recall from Definition \ref{def:real_tree} that we require $\BR$-trees to be separable and thus second countable. Thus, there exists a countable collection $\cA$ of open sets of diameter at most $\epsilon$ that cover $\cS$. It is a.s.\ the case that for every $U\in\cA$, if $q(U) = 0$ then $\{i\colon s_i\in U\}=\emptyset$. Consequently, the $\epsilon$-ball about $s_j$ a.s.\ has positive $q$-measure. Therefore there is a.s.\ some other sample $s_i$ with $d(s_i,s_j)<\epsilon$. Finally, $d((s_i\wedge s_j)_{\cS},s_j) < d(s_i,s_j) < \epsilon$. 
\end{proof}

We define
\begin{equation}
\begin{split}
 I^{\cS}(a) &:= \{j\in\BZ\colon a\in [[r,s_j]]_{\cS}\} \quad \text{for }a\in\cS\\
 \text{and} \quad I^{\cT}(b) &:= \{j\in\BN\colon b\in [[0,t_j]]_{\ell}\} \quad \text{for }b\in\T.
\end{split}
\end{equation}

\begin{lemma}\label{lem:fringe_formula}
 For $i,j,u,v\in\BN$ with $i\neq j$ and $u\neq v$, up to null events,
 \begin{equation}
  I^{\cS}((s_u\wedge s_v)_{\cS})\cap\BN
  		= (u\wedge v)\cap\BN 
  		= I^{\T}((t_u\wedge t_v)_{\ell}),\label{eq:bp_fringe_comp}
 \end{equation}
 \begin{equation}
  I^{\cS}(s_u)\cap\BN
  		= \BN\cap\bigcap_{k\in\BZ\setminus\{u\}} (u\wedge k) 
  		= I^{\T}(t_u),\label{eq:sample_fringe_comp}
 \end{equation}
 \begin{equation}
  \big\{(s_u\wedge s_v)_{\cS} = (s_i\wedge s_j)_{\cS}\big\}
  		= \big\{(t_u\wedge t_v)_{\ell} = (t_i\wedge t_j)_{\ell}\big\},\label{eq:same_bp_isom}
 \end{equation}
 \begin{equation}
  \text{and} \quad
  \{s_u = s_v\}
  		= \{t_u = t_v\}.\label{eq:same_sample_isom}
  		\quad\hphantom{and}
 \end{equation}
\end{lemma}

\begin{proof}
 \eqref{eq:bp_fringe_comp}: Note that for $u,v\in\BN$ distinct and $n> u,v$,
 \begin{equation*}
 \begin{split}
  (u\wedge v)_{n} = \bigcap_{A\in\cH_n\colon u,v\in A} A
  	&= \bigcap_{x\in\cS\colon s_u,s_v\in\fringe{x}{\cS}}\left( [\pm n]\cap I^{\cS}(x) \right)\\
  	&= [\pm n]\cap I^{\cS}((s_u\wedge s_v)_{\cS}),
 \end{split}
 \end{equation*}
 where the first equation is Definition \ref{def:MRCA} of the MRCA, the second follows from the definition of $\cH_n$ via the samples $(s_j)$, and the last follows because every fringe subtree containing both $s_u$ and $s_v$ must contain the branch point $(s_u\wedge s_v)_{\cS}$. This proves the first equation in \eqref{eq:bp_fringe_comp}. The second has already been established in Proposition \ref{prop:tree_limits}\ref{item:MRCA_equals_fringe}.

 \eqref{eq:sample_fringe_comp}: By Lemma \ref{lem:span_samp_approx},
 \begin{gather*}
  \fringe{s_u}{\cS} = \bigcap_{k\in\BZ\setminus\{u\}} \fringe{(s_u\wedge s_k)_{\cS}}{\cS};\\
  \text{thus,} \quad
  I^{\cS}(s_u)\cap\BN = \BN\cap\bigcap_{k\in\BZ\setminus\{u\}} I^{\cS}((s_u\wedge s_k)_{\cS}) = \BN\cap\bigcap_{k\in\BZ\setminus\{u\}} (u\wedge k)_{\cH},
 \end{gather*}
 with the last equation following from \eqref{eq:bp_fringe_comp}. By Proposition \ref{prop:tree_limits}\ref{item:driving}, the $(t_i,\,i\ge1)$ have $p$ as their driving measure, so the same argument via Lemma \ref{lem:span_samp_approx} applies to $I^{\T}(t_u)$, thus proving \eqref{eq:sample_fringe_comp}.
 
 \eqref{eq:same_bp_isom}: Note that $(s_i\wedge s_j)_{\cS} = (s_u\wedge s_v)_{\cS}$ if and only if both $i,j\in I^{\cS}((s_u\wedge s_v)_{\cS})$ and $u,v\in I^{\cS}((s_i\wedge s_j)_{\cS})$. The corresponding claim holds for samples in $\cT$. Thus, \eqref{eq:same_bp_isom} follows from \eqref{eq:bp_fringe_comp}.
 
 \eqref{eq:same_sample_isom}: Note that $s_u = s_v$ if and only if both $v\in I^{\cS}(s_u)$ and $u\in I^{\cS}(s_v)$. The corresponding claim holds for $t_u$ and $t_v$. Thus, \eqref{eq:same_sample_isom} follows from \eqref{eq:sample_fringe_comp}.
\end{proof}

\subsection{Proof of Proposition \ref{prop:MSE_to_constr_tree}}\label{sec:MSE_to_constr_tree}

We know from Theorem \ref{thm:hier_rep_constr} that $\Theta(\T,\ell_1,0,p) = \Theta(\cS,d,r,q)$ a.s.. Thus, it suffices to show that these two trees are a.s.\ mass-structurally equivalent. First, we will define a function $\phi$ mapping the special points of $\cS$, in the sense of Definition \ref{def:special_pts}, to those of $\cT$, and we show that it is a bijection. Then we will show that $\phi$ is mass and structure preserving.
 
 Recall that, by Proposition \ref{prop:its_a_bead_crush}, $(\cT,\ell_1,0,p)$ is an IP tree. In particular, it possesses the Spanning property, $\cT = \Span(p)$.\smallskip
 
 \textbf{Definition of a bijection, $\phi$}. Recall from Definition \ref{def:special_pts} that there are three types of special points: locations of atoms, branch points of the subtree spanned by the measure, and isolated leaves of said subtree. Therefore, we define a bijection $\phi$ in these three cases.
 
 \ref{item:special_atoms} If $y$ is the location of an atom of $q$ then there is a.s.\ some $i$ for which $s_i = y$. We define $\phi(y) := t_i$. By \eqref{eq:same_sample_isom}, it is a.s.\ the case that $t_j = t_i$ if and only if $s_j = s_i$, for $j\ge 1$, so this is well-defined. Moreover, we conclude from Proposition \ref{prop:tree_limits}\ref{item:driving} that $t_i$ is the location of an atom in $p$ with $p\{t_i\} = q\{s_i\}$. By the preceding argument, $\phi$ is injective from atoms of $q$ to those of $p$. The same argument in reverse shows that it bijects these sets of atoms.
 
 \ref{item:special_BPs} If $x$ is a branch point of $\Span(q)$, in the sense of Definition \ref{def:special_pts}, then there is a.s.\ some pair $i,j\in\BN$ for which $x = (s_i\wedge s_j)_{\cS}$ with $x\neq s_i$ and $x\neq s_j$. In particular, $s_i\notin [[r,s_j]]_{\cS}$ and vice versa. By \eqref{eq:sample_fringe_comp}, $t_i\notin [[r,t_j]]_{\cS}$ and vice versa, so $(t_i\wedge t_j)_{\ell}$ is a branch point of $\cT = \Span(p)$. And by \eqref{eq:same_bp_isom}, the vertex $(t_i\wedge t_j)_{\ell}$ is a.s.\ the same across all pairs $i,j$ for which $x = (s_i\wedge s_j)_{\cS}$. We define $\phi(x) := (t_i\wedge t_j)_{\ell}$. By this same argument in reverse, starting with a branch point of $\T$, we see that $\phi$ bijects the branch points of $\Span(q)$ with those of $\T$.
 
 In the special case that $q$ has an atom located at the branch point $x$, this agrees with our previous definition of $\phi$ for atoms. In this case, there exist samples $s_u = s_v = x$ with $u\neq v$. Then $(s_i\wedge s_j)_{\cS} = x = s_u = (s_u\wedge s_v)_{\cS}$. By \eqref{eq:same_bp_isom} this means $(s_i\wedge s_j)_{\cS} = (t_u\wedge t_v)_{\ell}$, and by \eqref{eq:same_sample_isom}, $t_u = t_v$. Then we conclude $\phi(x) = (t_i\wedge t_j)_{\ell} = t_u$.
 
 \ref{item:special_leaves} Now suppose $z\in\cS$ is an isolated leaf of $\Span(q)$, in the sense that there is a non-trivial segment $[[x,z]]_{\cS} \subseteq [[r,z]]_{\cS}$ that contains no branch points of $\Span(q)$ and every such segment has positive mass under $q$. Consider
 \begin{equation}\label{eq:iso_leaf_to_samples}
  J := \left\{ i\ge 1\ \middle|\ \forall j\in I^{\cS}(s_i),\; z\in \fringe{s_j}{\cS} \right\}.
 \end{equation}
 This is the set of indices of all samples that lie on a branch with the properties mentioned above. 
 The samples $(s_i,\,i\in J)$ all lie along $[[r,z]]_{\cS}$, and they are totally ordered, up to equality, along this segment. Since $z$ is in the closed support of $q$, it is the unique limit point of this set at maximal distance from $r$. By \eqref{eq:sample_fringe_comp}, the samples $(t_i,\,i\in J)$ are correspondingly totally ordered along a segment. As $\cT$ is bounded and complete under $\ell_1$, these samples also have a unique limit point $z'\in\cT$ at maximal distance from $0$. We define $\phi(z) := z'$.
 
 
 To show that this is a bijection between the sets of isolated leaves, we consider properties of the set $J$. It a.s.\ satisfies:\vspace{-2pt}
 \begin{enumerate}[label=(\roman*), ref=(\roman*)]
  \item $\forall i\in J,\ \BN\cap I^{\cS}(s_i) \subseteq J$ and
  \item $\forall i,j\in J,\ j\in I^{\cS}(s_i) \text{ and/or }i\in I^{\cS}(s_j).$\vspace{-2pt}
 \end{enumerate}
 Condition (i) asserts, roughly, that $J$ comprises indices of all samples that fall into some fringe subtree $B\subseteq\cS$. Condition (ii) asserts that these samples are totally ordered, up to equality, along a branch going away from $r$. I.e.\ the support of $q$ on $B$ is contained within a single segment aligned with $r$. By its definition, $J$ is maximal with these two properties. If we view \eqref{eq:iso_leaf_to_samples} as a map sending $z$ to $J$, then this is a bijection from isolated leaves of $\Span(q)$ to maximal sets of indices $J$ that satisfy properties (i) and (ii) above. Likewise,
 \begin{equation*}
  z' \mapsto \left\{ i\ge 1\ \middle|\ \forall j\in I^{\cT}(t_i),\; z'\in \fringe{t_j}{\cT} \right\}.
 \end{equation*}
 is a bijection from isolated leaves of $\cT$ to maximal sets $J$ satisfying:\vspace{-2pt}
 \begin{enumerate}[label=(\roman*'), ref=(\roman*')]
  \item $\forall i\in J,\ I^{\cT}(t_i) \subseteq J$ and
  \item $\forall i,j\in J,\ j\in I^{\cT}(t_i) \text{ and/or }i\in I^{\cT}(t_j)$.\vspace{-2pt}
 \end{enumerate}
 Finally, by \eqref{eq:sample_fringe_comp}, conditions (i) and (ii) are equivalent to (i') and (ii'). Therefore, $\phi$ bijects the isolated leaves of $\Span(q)$ with those of $\cT$.
 
 In the special case that $q$ has an atom at $z$, this again agrees with our previous definition of $\phi$ for atoms. In this case, there exists some $i$ with $s_i = z$. Since $z$ is a leaf of $\Span(q)$, $i\in J$ and $s_i$ is the least upper bound of samples $(s_j,\,j\in J)$. By \eqref{eq:sample_fringe_comp}, $t_i$ is then the least upper bound of samples $(t_j,\,j\in J)$. Thus, $\phi(z) = t_i$, as desired.\smallskip
 
 \textbf{Mass preserving}. We have already established that $q\{x\} = p\{\phi(x)\}$ for all points $x\in\cS$ at which $q$ has atoms, and that $\phi$ bijects the locations of atoms of $q$ with those of $p$.
 
 For $j\ge 1$, it is a.s.\ the case that
 \begin{equation*}
 \begin{split}
  q([[r,s_j]]_{\cS}) &= \lim_{n\to\infty} \frac{\#\{i\in [n]\colon s_i\in [[r,s_j]]_{\cS}\}}{n}\\
  		&= \lim_{n\to\infty} \frac{\#\{i\in [n]\colon j\in \cI^{\cS}(i)\}}{n} = \lim_{n\to\infty}\frac{\#\{i\in [n]\colon j\in \cI^{\cT}(i)\}}{n}\\
  		&= \lim_{n\to\infty} \frac{\#\{i\in [n]\colon t_i\in [[r,t_j]]_{\ell}\}}{n} = p([[0,t_j]]_{\ell}),
 \end{split}
 \end{equation*}
 with the first and last equations a consequence of $q$ and $p$ being driving measures for the $(s_i)$ and $(t_i)$, respectively; the second and fourth following from the definition of fringe subtrees; and the third following from \eqref{eq:sample_fringe_comp}. An analogous derivation, making use of \eqref{eq:bp_fringe_comp} in place of \eqref{eq:sample_fringe_comp}, shows that $q([[r,(s_i\wedge s_j)_{\cS}]]_{\cS}) = p([[0,(t_i\wedge t_j)_{\ell}]]_{\ell})$. This proves that $q([[r,x]]_{\cS}) = p([[0,\phi(x)]]_{\ell})$ when $x$ is the location of an atom of $q$ or a branch point of $\Span(q)$. Finally, the map $x\mapsto q([[r,x]]_{\cS})$ is continuous at points $x$ that are neither branch points nor locations of atoms of $q$, and correspondingly for $p$. Thus, by passing through a limit with samples converging to an isolated leaf, the result also holds when $x$ is an isolated leaf of $\Span(q)$.
 
 If $z$ is an isolated leaf of $q$ at which there is no atom, then 
 $q\left(\fringe{z}{\cS}\right) = 0 = p\left(\fringe{\phi(z)}{\cT}\right)$. Finally, for $y$ a branch point of $\Span(q)$ or the location of an atom of $q$, we can write $y = (s_i\wedge s_j)_{\cS}$ for some $1\le i<j$. Then, by \eqref{eq:bp_fringe_comp},
 \begin{equation*}
 \begin{split}
  q\left(\fringe{(s_i\wedge s_j)_{\cS}}{\cS}\right) &= \lim_{n\to\infty} n^{-1}\# \left( I^{\cS}((s_i\wedge s_j)_{\cS}) \cap [n] \right)\\
  		&= \lim_{n\to\infty} n^{-1}\# \left( I^{\cT}((t_i\wedge t_j)_{\ell}) \cap [n] \right) = p\left(\fringe{(t_i\wedge t_j)_{\ell}}{\cT}\right),
 \end{split}
 \end{equation*}
 as desired.\smallskip
 
 
 \textbf{Structure preserving}. We must confirm that structure is preserved, in the sense of Definition \ref{def:mass_struct}\ref{item:m_s:s}, between any two special points in $\cS$. Again, we approach this case-by-case for the different types of special points.
 
 For branch points $y_1$ and $y_2$ of $\Span(q)$, we have $y_1 = (s_i\wedge s_j)_{\cS}$ and $y_2 = (s_u\wedge s_v)_{\cS}$ for some $i,j,u,v\in\BN$. Then by \eqref{eq:bp_fringe_comp} and the definition of $(a\wedge b)_{\cS}$,
 \begin{equation*}
 \begin{split}
  (s_i\wedge s_j)_{\cS} \in [[r,(s_u\wedge s_v)_{\cS}]]_{\cS}\ &\iff\ u,v\in I^{\cS}((s_i\wedge s_j)_{\cS})\\
  		&\iff\ u,v\in I^{\cT}((t_i\wedge t_j)_{\ell})\\
  		&\iff\ (t_i\wedge t_j)_{\ell}\in [[0,(t_u\wedge t_v)_{\ell}]]_{\ell}.
 \end{split}
 \end{equation*}
 The same argument shows that $\phi$ preserves structure between two locations of atoms $x_1,x_2$, or between a branch point and an atom, by taking $s_i = s_j = x_1$ for some pair $i\neq j$ so that $(s_i\wedge s_j)_{\cS} = x_1$, and correspondingly for $x_2$.
 
 If $z_1$ and $z_2$ are both isolated leaves of $\Span(q)$ then $z_1\notin [[r,z_2]]_{\cS}$ and $z_2\notin [[r,z_1]]_{\cS}$, since both are leaves of the same tree, and likewise for $\phi(z_1)$ and $\phi(z_2)$. Thus, structure is preserved here as well.
 
 Finally, suppose that $z$ is an isolated leaf of $\Span(q)$ with $q\{z\} = 0$ and $x$ is either the location of an atom of $q$ or a branch point of $\Span(q)$. In either case, $x = (s_i\wedge s_j)_{\cS}$ for some distinct $i,j\in\BN$. We cannot have $z\in [[r,x]]_{\cS}$, nor can we have $\phi(z)\in [[0,\phi(x)]]_{\ell}$, since $z$ and $\phi(z)$ are leaves and do not equal $x$ or $\phi(x)$, respectively. Let $J$ be as in \eqref{eq:iso_leaf_to_samples}. Then
 \begin{equation*}
 \begin{split}
  (s_i\wedge s_j)_{\cS} \in [[r,z]]_{\cS}\ &\iff\ I^{\cS}((s_i\wedge s_j)_{\cS}) \cap J\neq \emptyset\\
  		&\iff\ I^{\cT}((t_i\wedge t_j)_{\ell}) \cap J\neq \emptyset\ \iff\ (t_i\wedge t_j)_{\ell}\in [[0,\phi(z)]]_{\ell}.
 \end{split}
 \end{equation*}
 Thus, $\phi$ preserves structure between isolated leaves of $\Span(q)$ and other special points. \qed

\subsection{Proof of Proposition \ref{prop:IP_tree_MSE}}\label{sec:IP_tree_MSE}

 Let $(\cT_i,d_i,r_i,p_i)$ for $i=1,2$ be a pair of IP trees, with special point sets $\scS_1$ and $\scS_2$ and $\phi\colon \scS_1\to\scS_2$ a mass-structural isomorphism. We begin with a pair of observations.
 
 First, the roots $r_1$ and $r_2$ need not be special points. However, for $x\in\scS_1$,
 \begin{equation}
  d_2(r_2,\phi(x)) = 1-p_2\left(\fringe{\phi(x)}{\cT_2}\right) = 1-p_1\left(\fringe{x}{\cT_1}\right) = d_1(r_1,x),
 \end{equation}
 by the Spacing properties of the two IP trees and the mass preserving property of $\phi$. Taking $x = r_1$ or $\phi(x) = r_2$ shows that $r_1$ is a special point if and only if $r_2$ is, in which case $\phi(r_1) = \phi(r_2)$. If they are not special points, then we define $\phi(r_1) := r_2$.
 
 Second, since $\scS_1$ and $\scS_2$ contain all branch points of the two trees, it follows from the structure preserving property that $\phi((x\wedge y)_{\cT_1}) = (\phi(x),\phi(y))_{\cT_2}$ for every $x,y\in\scS_1$. Thus,
 \begin{equation*}
 \begin{split}
  d_1(x,y) &= d_1\big(x,(x\wedge y)_{\cT_1})\big) + d_1\big((x\wedge y)_{\cT_1}),y\big)\\
  		&= 2p\big( \fringe{(x\wedge y)_{\cT_1}}{\T_1} \big) - p\big( \fringe{x}{\T_1} \big) - p\big( \fringe{y}{\T_1} \big)\\
  		&= 2p\big( \fringe{\phi((x\wedge y)_{\cT_1})}{\T_2} \big) - p\big( \fringe{\phi(x)}{\T_2} \big) - p\big( \fringe{\phi(y)}{\T_2} \big)\\
  		&= d_2\big(\phi(x),\phi((x\wedge y)_{\cT_1})\big) + d_2\big(\phi((x\wedge y)_{\cT_1}),\phi(y)\big) = d_2(\phi(x),\phi(y)),
 \end{split}
 \end{equation*}
 where the second and fourth lines follow from the Spacing properties of $\T_1$ and $\T_2$ and the third is an application of the mass preserving property of $\phi$.  In other words, $\phi$ is an isometry from $(\scS_1\cup\{r_1\},d_1)$ to $(\scS_2\cup\{r_2\},d_2)$.
 
 We must show that the IP trees $(\cT_i,d_i,r_i,p_i)$ for $i=1,2$ are isomorphic. First, we will define a map $\psi\colon \cT_1\to\cT_2$ that preserves distance from the root; then, we show that $\psi$ is an isometry; and finally we prove that $\psi$ is measure-preserving.\smallskip
 
 \textbf{Definition of $\psi$}. We extend $\phi$ to define $\psi\colon \T_1 \to \cT_2$ by two mechanisms, which we call overshooting and approximation. Consider $z\in\cT_1\setminus\scS_1$.
 
 Case 1 (overshooting): $\fringe{z}{\cT_1}\cap\scS_1 \neq \emptyset$. Consider $x\in\fringe{z}{\cT_1}\cap\scS_1$. Define $\psi(z)$ to be the point along $[[r_2,\phi(x)]]_{\cT_2}$ at distance $d_1(r_1,z)$ from $r_2$. This definition does not depend on our choice of $x$: if $x_1,x_2\in \fringe{z}{\cT_1}\cap\scS_1$ then $x^* := (x_1\wedge x_2)_{\cT}\in \fringe{z}{\cT_1}\cap\scS_1$ as well. In that case, $d_1(r_1,z) < d_1(r_1,x^*) = d_2(r_2,\phi(x^*))$, and by the structure preserving property of $\phi$,
 $$[[r_2,\phi(x^*)]]_{\cT_2} = [[r_2,\phi(x_1)]]_{\cT_2}\cap [[r_2,\phi(x_2)]]_{\cT_2}.$$
 Thus, the points along $[[r_2,\phi(x_i)]]_{\cT_2}$ at distance $d_1(r_1,z)$ from $r_2$ are the same for $i=1,2$, as both lie in $[[r_2,\phi(x^*)]]_{\cT_2}$.
 
 Case 2 (approximation): $\fringe{z}{\cT_1}\cap\scS_1 = \emptyset$. Then there is no branch point, nor any isolated leaf of $\Span(p_1) = \cT_1$ beyond $z$. Thus, $z$ must be a leaf with a sequence of branch points $(x_i,\,i\ge1)$ converging to it along $[[r_1,z]]_{\cT_1}$. Moreover, since $z$ is a leaf and not the location of an atom, $d_1(r_1,z) = 1$ by the Spacing property. Since $\phi$ is an isometry, the sequence $(\phi(x_i),\,i\ge 1)$ is a Cauchy sequence in $d_2$, so it has a limit $z'$ with $d_2(r_2,z')=1$. We define $\phi(z) := z'$. Again, this is well-defined. If $(y_i,\,i\ge1)$ is another sequence of branch points converging to $z$, then so is $x_1,y_1,x_2,y_2,\ldots$, so the $\phi$-images of these sequences must have the same limit.
 
 Note that $\psi$ preserves distance from the root, by definition. Moreover, if $z$ is defined by approximation then
 \begin{equation}\label{eq:BP_before_def_by_approx}
  d_2(\phi(x),\phi(z)) = d_1(x,z) \qquad \text{for branch points }x\in [[r_1,z]]_{\cT_1}.
 \end{equation}
 
 \textbf{Isometry}. It follows from Lemma \ref{lem:span_samp_approx} and the definition above that $\psi$ is a surjection. The definition also implies that $\psi$ preserves distance from the root. Thus, to show that it is an isometry, it suffices to show that it preserves structure, in the sense that $\psi(x)\in [[r_2,\psi(y)]]_{\cT_2}$ if and only if $x\in [[r_1,y]]_{\cT_1}$. We consider two cases in which $x\in [[r_1,y]]_{\cT_1}$ and one in which $x\notin [[r_1,y]]_{\cT_1}$.
 
 Case A.I: $x\in [[r_1,y]]_{\cT_1}$ and $y\in [[r_1,z]]_{\cT_1}$ for some $z\in\scS_1$. Then both $\psi(x)$ and $\psi(y)$ lie on $[[r_2,\psi(z)]]_{\cT_2}$, at respective distances $d_1(r_1,x)$ and $d_1(r_1,y)$ from $r_2$. Since $d_1(r_1,x) \le d_1(r_1,y)$, we get $\psi(x) \in [[r_2,\psi(y)]]_{\cT_2}$, as desired.
 
 Case A.II: $x\in [[r_1,y]]_{\cT_1}$ and $\psi(y)$ is defined by approximation. This means that we can take $z\in [[r_1,y]]$ to be a branch point with $d_1(z,y) < d_1(x,y)/2$. Then $z$ must belong to $\fringe{x}{\cT_1}$, so by the definition of $\psi(x)$ by overshooting, $\psi(x)\in [[r_2,\psi(z)]]_{\cT_2}$.  
 Moreover,
 $$d_2(\psi(x),\psi(z)) = d_1(x,z) > d_1(z,y) = d_2(\psi(z),\psi(y)),$$
 with the first equation following from preservation of distance from the root, the inequality from our assumption that $d_1(z,y) < d_1(x,y)/2$, and the final equation from \eqref{eq:BP_before_def_by_approx}. The entire closed ball of radius $d_2(\psi(x),\psi(z))$ about $\psi(z)$ lies inside $\fringe{\psi(x)}{\cT_2}$. In particular, $\psi(y)\in\fringe{\psi(x)}{\cT_2}$, as desired.
 
 Case B: $x\notin [[r_1,y]]_{\cT_1}$ and $y\notin [[r_1,x]]_{\cT_1}$. We take up the case in which $\psi(x)$ is defined by overshooting and $\psi(y)$ by approximation; the other cases can be addressed similarly. Let $z$ be a special point in $\fringe{x}{\cT_1}$ and $z'$ a branch point in $[[r_1,y]]_{\cT_1}$ with $d_1(z',y) < d_1((x\wedge y)_{\cT_1},y) / 2$. Then $(z\wedge z')_{\cT_1} = (x\wedge y)_{\cT_1}$. Moreover, by the structure-preserving property of $\phi$,
 $$\psi((x\wedge y)_{\cT_1}) = \phi((z\wedge z')_{\cT_1}) = (\phi(z)\wedge \phi(z'))_{\cT_2} = (\psi(z)\wedge \psi(z'))_{\cT_2}.$$
 By definition,
 $$d_2(\psi(x),\psi(z)) = d_1(x,z) < d_1(z,(x\wedge y)_{\cT_1}) = d_2\big(\psi(z),\psi((x\wedge y)_{\cT_1})\big).$$
 Thus, $\psi(x)$ is in the component of $\fringe{\psi((x\wedge y)_{\cT_1})}{\cT_2}\setminus\{\psi((x\wedge y)_{\cT_1})\}$ that contains $\psi(z)$. Correspondingly,
 $$d_2(\psi(y),\psi(z')) = d_1(y,z') < d_1(z',(x\wedge y)_{\cT_1}) = d_2(z',\psi((z\wedge z')_{\cT_1})).$$
 Thus, $\psi(y)$ is in the component of $\fringe{\psi((x\wedge y)_{\cT_1})}{\cT_2}\setminus\{\psi((x\wedge y)_{\cT_1})\}$ that contains $\psi(z')$. We conclude that $\psi(x)\notin [[r_2,\psi(y)]]_{\cT_2}$ and vice versa, as desired.
 \smallskip
 
 \textbf{Measure-preserving}. The fringe subtrees of $\T_1$ comprise a $\pi$-system that generates the Borel $\sigma$ algebra on $\T_1$, and likewise for $\T_2$. Because $\psi$ is a root-preserving isometry, for $x\in\cT_1$, $\psi\big(\fringe{x}{\cT_1}\big) = \fringe{\psi(x)}{\cT_2}$. Thus, by a monotone class argument, it suffices to show that for every $x\in\cT_1$, $p_1\left( \fringe{x}{\cT_1} \right) = p_2\left( \fringe{\psi(x)}{\cT_2} \right)$. We argue this in four cases.
 
 Case 1: $x = r_1$. Then $p_1\left( \fringe{x}{\cT_1} \right) = 1 = p_2\left( \fringe{\psi(x)}{\cT_2} \right)$.
 
 Case 2: $x\in\scS_1$. Then $\psi(x) = \phi(x)$, and the desired equality is exactly the Mass preserving property of $\phi$.
 
 Case 3: $x$ is not special but is the limit of a sequence of special points $(x_i,\,i\ge 1)$ in $[[r_1,x]]_{\cT_1}\cup\fringe{x}{\cT_1}$. Then $x$ is neither a branch point nor the location of an atom, and likewise for $\psi(x)$, so
 $$p_1 \left(\fringe{x}{\cT_1}\right) = \lim_{i\to\infty}p_1 \left(\fringe{x_i}{\cT_1}\right) = \lim_{i\to\infty}p_2\left(\fringe{\psi(x_i)}{\cT_2}\right) = p_2\left(\fringe{\psi(x)}{\cT_2}\right).$$
 
 Case 4: $x$ is not special and is not a limit of special points. Then $x$ cannot be a leaf. Let $y$ and $z$ be the points closest to $x$ in $(\cl(\scS_1)\cup\{r_1\})\cap [[r_1,x]]_{\cT_1}$ and $\cl(\scS_1)\cap\fringe{x}{\cT_1}$, respectively. The map $w\mapsto p_1([[r_1,w]]_{\cT_1})$ is continuous except at locations of atoms of $p_1$, and correspondingly for $p_2$. By the Mass preserving property of $\phi$, the isometry property of $\psi$, and this continuity,
 \begin{equation*}
 \begin{split}
  M &:= p_1\big([[y,z]]_{\cT_1}\setminus\{z\}\big) = p_1([[r_1,z]]_{\cT_1}) - p_1\{z\} - p_1([[r_1,y]]_{\cT_1})\\
  		&\hphantom{:}= p_2\big([[r_2,\psi(z)]]_{\cT_2}\big) - p_2\{\psi(z)\} - p_2\big([[r_2,\psi(y)]]_{\cT_2}\big)\\
  		&\hphantom{:}= p_2\big([[\psi(y),\psi(z)]]_{\cT_2}\setminus\{\psi(z)\}\big)
 \end{split}
 \end{equation*}
 By the Spacing property, $d_1(y,z) = p_1\left(\fringe{y}{\cT_1}\right) - p_1\left(\fringe{z}{\cT_1}\right) \ge M$. 
 Let $v$ be the point in $[[y,z]]_{\cT_1}$ at distance $M$ from $z$. Then the Spacing property of $\cT_1$ implies that $p_1$ is null on $[[y,v]]_{\cT_1}\setminus\{y\}$ and equals length measure on $[[v,z]]_{\cT_1}\setminus\{z\}$. Correspondingly, the Spacing property of $\T_2$ implies that $p_2$ is null on $[[\psi(y),\psi(v)]]_{\T_2}$ and equals length measure on $[[\psi(v),\psi(z)]]_{\cT_2}\setminus\{\psi(z)\}$. In particular,
 \begin{equation*}
 \begin{split}
  p_1(\fringe{x}{\cT_1}) &= p_1([[x,z]]_{\cT_1}\setminus\{z\}) + p_1(\fringe{z}{\cT_1})\\
  &= \min\{d_1(x,z),M\} + p_1\left(\fringe{z}{\cT_1}\right)\\
  &= \min\big\{d_2(\psi(x),\psi(z)),M\big\} + p_2\big(\fringe{\psi(z)}{\cT_2}\big)
  = p_2\big(\fringe{\psi(x)}{\cT_2}\big).\!\!\!\!\!\qed
 \end{split}
 \end{equation*}
 

\section{Proofs of theorems}\label{sec:thm_pfs}

\begin{proof}[Proof of Theorem \ref{thm:MS_rep}]
 Consider a rooted, weighted $\BR$-tree $(\cT,d,r,p)$. By Proposition \ref{prop:MSE_to_constr_tree}, it is mass-structurally equivalent to at least one IP tree. By Proposition \ref{prop:IP_tree_MSE}, all such IP trees are isomorphic to each other.
\end{proof}

\begin{proof}[Proof of Theorem \ref{thm:bead_crush_rep}]
 Consider an IP tree $(\cT,d,r,p)$. By Proposition \ref{prop:MSE_to_constr_tree}, there exists an IP tree arising from a deterministic bead crushing construction that is mass-structurally equivalent to $(\cT,d,r,p)$. By Proposition \ref{prop:IP_tree_MSE}, the two IP trees are thus isomorphic.
\end{proof}

\begin{proof}[Proof of Theorem \ref{thm:measure_decomp}]
 By Theorem \ref{thm:bead_crush_rep}, it suffices to prove this theorem for IP trees that arise from bead crushing constructions. Consider such an IP tree $(\T,\ell_1,0,p)$ constructed, as in Section \ref{sec:bead_crush}, from a sequence of uniformized probability measures $q_n$, $n\ge 1$. We need only show that the restriction of the non-atomic component of $p$ to the skeleton of $\T$ equals the restriction of the length measure to a subset of the skeleton.
 
 In Section \ref{sec:bead_crush}, in between \eqref{eq:bead_crush} and Proposition \ref{prop:bead_crush_IP}, we note that the sequence of measures $(p_n,\,n\ge1)$ arising in the construction is projectively consistent, $p_n = \pi_n(p_{n+1})$ for $n\ge 1$, and we define the limiting measure $p$ via the Daniell-Kolmogorov extension theorem. The skeleton of the tree contains only points in $\ell_1$ with finitely many positive coordinates. Thus, the diffuse component of the measure on the skeleton, $p^s$, is the sum over $n$ of the diffuse measure on the $n^{\text{th}}$ branch added in the construction. To get the diffuse measure on such a branch, the construction takes the diffuse component of $q_n$ -- call it $q_n^d$ -- scales down both its total mass and the length of the segment supporting it by some factor $a_n\in (0,1)$, i.e.\ $q_n^d \mapsto a_nq^d_n(\,\cdot\,/a_n)$, and it transposes the measure from the line segment to the branch in $\cT$. Now, the theorem follows from Lemma \ref{lem:uniformization}, which states that $q_n^d$ is the restriction of Lebesgue measure to a subset of $[0,1]$.
\end{proof}

\begin{proposition}\label{prop:Theta_isom}
 Two IP trees $(\cT_i,d_i,r_i,p_i)$, $i=1,2$, are isomorphic if and only if $\Theta(\cT_1,d_1,r_1,p_1) = \Theta(\cT_2,d_2,r_2,p_2)$.
\end{proposition}

\begin{proof}
 We have already mentioned, and it is trivial to check, that isomorphic trees have the same image under $\Theta$. Now, suppose that the two trees have the same image under $\Theta$. Let $(\cH_n,\,n\ge 1)$ be an exchangeable random hierarchy with law $\Theta(\cT_1,d_1,r_1,p_1)$. Let $(\cT,\ell_1,0,p)$ denote the random IP tree representation of $(\cH_n)$ obtained from the construction in Section \ref{sec:tree_construction}. By Proposition \ref{prop:MSE_to_constr_tree}, all three IP trees are a.s.\ mass-structurally equivalent. Then, by Proposition \ref{prop:IP_tree_MSE} they are a.s.\ isomorphic. In particular, the two deterministic trees must be isomorphic.
\end{proof}

\begin{proof}[Proof of Theorem \ref{thm:MS_Theta}]
 First, suppose that $(\cT_i,d_i,r_i,p_i)$, $i=1,2$, are two rooted, weighted $\BR$-trees with the same image as each other under $\Theta$. Then by the same argument as in the proof of Proposition \ref{prop:Theta_isom}, they must be mass-structurally equivalent to each other.
 
 Now, suppose instead that the two rooted, weighted $\BR$-trees are mass-structurally equivalent. By Proposition \ref{prop:MSE_to_constr_tree}, each tree $(\cT_i,d_i,r_i,p_i)$ is then mass-structurally equivalent to some IP tree $(\cS_i,\ell_1,0,q_i)$ for which $\Theta(\cT_i,d_i,r_i,p_i) = \Theta(\cS_i,\ell_1,0,q_i)$. By the transitivity of mass-structural equivalence, the two IP trees are mass-structurally equivalent. By Proposition \ref{prop:IP_tree_MSE}, that means the IP trees are isomorphic, so
 $$\Theta(\cT_1,d_1,r_1,p_1) = \Theta(\cS_1,\ell_1,0,q_1) = \Theta(\cS_2,\ell_1,0,q_2) = \Theta(\cT_2,d_2,r_2,p_2).\vspace{-21pt}$$
\end{proof}

\begin{proof}[Proof of Theorem \ref{thm:hierarchies_IP}]
 \ref{item:hIP:rcd} Let $(\cH_n,\,n\ge1)$ be an exchangeable random hierarchy. By Theorem \ref{thm:hier_rep_constr} and Proposition \ref{prop:its_a_bead_crush}, there exists a random IP tree $(\cT,\ell_1,0,p)$ with the property that $\Theta(\cT,\ell_1,0,p)$ is a r.c.d.\ for $(\cH_n,\,n\ge1)$ given its tail $\sigma$-algebra, $\tail(\cH_n)$. Since $\Theta(\cT,\ell_1,0,p)$ is $\tail(\cH_n)$-measurable, it follows from Proposition \ref{prop:Theta_isom} that the random isomorphism class $\scT$ of $(\cT,\ell_1,0,p)$ is as well. Then $\Theta(\scT) = \Theta(\cT,\ell_1,0,p)$ is a r.c.d.\ for $(\cH_n)$.
 
 \ref{item:hIP:bijection} Proposition \ref{prop:Theta_isom} states that the map $\Theta$ from isomorphism classes of IP trees to e.i.g.\ hierarchy laws is injective. By Theorem \ref{thm:hier_rep_constr} and Proposition \ref{prop:its_a_bead_crush}, every e.i.g.\ law is the $\Theta$ image of an IP tree, so it is also surjective.
\end{proof}


\section{Complements}\label{sec:complements}

\subsection{Recovering the Brownian CRT from a Brownian IP tree}\label{sec:ATCRT_recovery}

In light of Theorem \ref{thm:MS_rep}, for each isomorphism class $\scT$ of rooted, weighted $\BR$-trees, there is a single isomorphism class $\scT'$ comprising the IP trees that are mass-structurally equivalent to the trees in $\scT$. Let $\Psi$ denote this map, from isomorphism classes of rooted, weighted $\BR$-trees to isomorphism classes of IP trees. This map is surjective but not injective. However, for certain interesting classes of CRTs $(\T,d,r,p)$, there exist sets $A$ such that: (i) $(\T,d,r,p)\in A$ a.s., and (ii) the restriction of $\Psi$ to isomorphism classes that intersect $A$ is injective. Property (ii) is equivalent to the condition that pairs of trees in $A$ are mass-structurally equivalent if and only if they are isomorphic. In particular, this holds for the Brownian CRT (and for all of the $(\alpha,\theta)$-trees of \cite{Ford05,PitmWink09}, though we will focus on the Brownian case). 

\begin{proposition}\label{prop:BIPT}
 It is possible to construct a Brownian CRT and Brownian IP tree, in the sense of Definition \ref{def:AT_IP_tree}, defined on a common probability space, such that they are a.s.\ mass-structurally equivalent.
\end{proposition}

\begin{proof}
Following \cite{PitmWink09,RembWinkString}, we can construct a Brownian CRT via a bead crushing construction similar to that in Section \ref{sec:bead_crush}. In fact, we will construct a coupled Brownian CRT and Brownian IP tree.

Let $(q_n,\,n\ge 1)$ denote an i.i.d.\ sequence of $\big(\frac12,\frac12\big)$-strings of beads, as described in Example \ref{eg:aa_string}. For each $n$, denote by $L_n$ the maximum of the support of $q_n$; this is a.s.\ finite. As in Section \ref{sec:bead_crush}, we define $\T_0 := \{0\}$ and $p_0 := \delta_0$ and proceed recursively to construct a tree embedded in $\ell_1$.

Assume $(\T_n,0,\ell_1,p_n)$ is a rooted, weighted real tree embedded in the first $n$ coordinates in $\ell_1$, with $p_n$ a purely atomic measure. Let $X_n$ be a sample from $p_n$, so $p_n(X_n) =: M_n > 0$. I.e.\ $M_n\delta_{X_n}$ is a size-biased random atom of $p_n$. Set
\begin{equation}\label{eq:bead_crush_2}
\begin{split}
 \phi_n(z) &:= X_n + z \sqrt{M_n}\Be_{n+1} \quad \text{for }z\in [0,L_{n+1}],\\
 \T_{n+1} &:= \T_n\cup \phi_n[0,L_{n+1}] = \T_n\cup [[X_n,\phi_n(L_{n+1})]]_{\ell},\\
 p_{n+1} &:= p_n + M_n\left(-\delta_{X_n} + \phi_n\left(q_{n+1}\right)\right),
\end{split}
\end{equation}
where $\phi_n(q_{n+1})$ denotes the pushforward of the measure. As in Section \ref{sec:bead_crush}, $p_n = \pi_n(p_N)$ for $N>n$, so again, by the Daniell-Kolmogorov extension theorem, there exists a measure $p$ on $\ell_1$ with $\pi_n(p) = p_n$ for $n\ge1$. Setting $\T := \cl(\bigcup_{n \geq 1} \T_n)$, the tree $(\cT,\ell_1,0,p)$ is a Brownian CRT. 

We now construct a Brownian IP tree coupled with this Brownian CRT. Let $q'_n$ to be the uniformization of $q_n$, in the sense of Definition \ref{def:uniformize}, for $n\ge 1$. There is a natural bijection from atoms of $q_n$ to those of $q'_n$ -- in fact, this bijection is a mass-structural isomorphism from $([0,L_n],d,0,q_n)$ to $([0,1],d,0,q'_n)$. We plug the measures $(q'_n,\,n\ge1)$ into the bead-crushing construction of Section \ref{sec:bead_crush} to recursively construct trees $(\cT'_n,\ell_1,0,p'_n)$. We see inductively that at each step, this resulting IP tree is mass-structurally equivalent to $(\T_n,\ell_1,0,p_n)$ from the other construction, and so to proceed to the next step we can crush an atom $X'_n\delta_{M_n}$ of $p'_n$ that corresponds to the atom $X_n\delta_{M_n}$ that was crushed in the other construction. In particular, this choice of $X'_n$ is a sample from $p'_n$. The resulting limiting tree $(\T',\ell_1,0,p')$ is a Brownian IP tree, as in Definition \ref{def:AT_IP_tree}.

Both $p$ and $p'$ are diffuse measures supported on the leaves of $\T$ and $\T'$, respectively. It follows from our inductive argument that there is a mass- and structure-preserving bijection from branch points of $\T$ to those of $\T'$. Thus, the two trees are mass-structurally equivalent.
\end{proof}

\begin{proposition}\label{prop:BCRT_unique}
 There exists a set of rooted, weighted $\BR$-trees $A$ with the properties that: (1) the Brownian CRT is a.s.\ isomorphic to a tree in $A$, and (2) no two trees in $A$ are mass-structurally equivalent.
\end{proposition}
Informally, this proposition states that the Brownian CRT is a.s.\ uniquely specified, up to isomorphism, by its mass-structural equivalence class.

\begin{proof}
 We prove this by constructing a Brownian CRT $(\T^*,d^*,r^*,p^*)$ as a deterministic function $\bar\psi$ of a Brownian IP tree $(\T',d',r',p')$ (conditioned on certain a.s.\ properties), in such a way that $\bar\psi$ sends isomorphic IP trees to isomorphic rooted, weighted $\BR$-trees. Once we have made this construction, then we view $(\T^*,d^*,r^*,p^*)$ as a function from a certain a.s.\ event to the space of rooted, weighted $\BR$-trees. The proposition is proved by taking $A$ to be a set of representatives of the isomorphism classes of trees with members appearing in the range of $(\T^*,d^*,r^*,p^*)$.
 
 Consider a general rooted, weighted $\BR$-tree $(\T,d,r,p)$. For $x,y\in\T$ with $x\in [[r,y]]_{\T}\setminus\{y\}$, we will write 
$$\bush{x\from y}{\T} := \fringe{x}{\T}\setminus\bigcup_{z\in [[x,y]]_{\cT}\setminus\{x\}}\fringe{z}{\T}.$$
Note that $\bush{x\from y}{\T} = \{x\}$ if and only if $x$ is not a branch point. In the language of Section \ref{sec:string_of_beads}, this is the \emph{bush} that branches off of the \emph{spine} $[[r,y]]_{\T}$ at $x$. We adopt the convention $\bush{y\from y}{\T} := \fringe{y}{\T}$. We define a purely atomic probability measure $p_y$ on the spine by $p_y\{x\} := p(\bush{x\from y}{\T})$.
 
 Let $(\T,d,r,p)$ and $(\T',d',r',p')$ be a Brownian CRT and Brownian IP tree, coupled as described in Proposition \ref{prop:BIPT}, so that there is a random mass-structural isomorphism $\phi$ that bijects the branch points of $\T$ with those of $\T'$. It follows from results in \cite{PitmWink09} that it is a.s.\ the case that for every $y\in\T$,
 \begin{equation}\label{eq:BCRT_LT_limits}
  d(r,y) = \sqrt{\pi}\lim_{h\to 0} \sqrt{h}\#\big\{x\in [[r,y]]_{\cT}\colon p_y\{x\} > h\big\}.
 \end{equation}
 Let $E$ denote an a.s.\ event on which this formula holds at every $y\in\T$ and $\T$ is binary.
 
 We will show that, on $E$, the following limit converges, for every $y\in\T'$:
 \begin{equation}
  f(y) := \sqrt{\pi}\lim_{h\to 0} \sqrt{h}\#\big\{x\in [[r',y]]_{\cT'}\colon p'_y\{x\} > h\big\}.
 \end{equation}
 Clearly $f(r') = 0$. For a branch point $y\in\T'$, this equals the corresponding limit in \eqref{eq:BCRT_LT_limits} for $\phi^{-1}(y)$, so it does indeed converge. For $y_1\neq y_2$ with $y_1\in [[r',y_2]]_{\cT'}$, the limit for $y_2$ must be strictly greater than that for $y_1$. 
 Consequently, for a point $y\in\T'$ that is neither a branch point nor the root,
 $$f(y) = \sup\{f(z)\colon z\ \text{is a branch point and lies in }[[r',y]]_{\cT'}\}.$$
 Thus, this limit exists as well. We define a semi-metric on $\T'$ by
 $$d^*(y,z) = f(y) + f(z) - 2f((y\wedge z)_{\cT'}) \qquad \text{for }y,z\in\T'.$$
 
 Let $\T^*$ denote the set of equivalence classes of points in $\T'$ under the relation: $x\sim y$ if $d^*(x,y) = 0$. Then $(\T^*,d^*)$ is a $\BR$-tree. Let $r^*$ denote the $\sim$-equivalence class containing $r'$, and let $p^*$ denote the push-forward of $p'$ via the map from $\T'$ to $\T^*$. In the event $E$, no two branch points of $\T'$ belong to the same $\sim$-equivalence class. Thus, the map from a branch point of $\T'$ to its $\sim$-class is a mass-structural isomorphism from $(\T',d',r',p')$ to $(\T^*d^*,r^*,p^*)$; call it $\psi$. Moreover, by \eqref{eq:BCRT_LT_limits} and by definition of $d^*$, $\psi\circ\phi$ is a mass-structural isomorphism that is also an isometry from the branch points of $\T$ to those of $\T^*$. The branch points are well known to be dense in the Brownian CRT (and indeed, this is implied by \eqref{eq:BCRT_LT_limits}). Thus, we conclude that $(\T^*,d^*,r^*,p^*)$ is a Brownian CRT that is a.s.\ isomorphic to $(\T,d,r,p)$ via the unique continuous extension of $\psi\circ\phi$.
\end{proof}


\subsection{Structural equivalence}\label{sec:struct_equiv}

In the introduction to this paper, we heuristically described mass-structural equivalence as equivalence of the interaction between mass and ``underlying tree structure.'' One notion of underlying structure was considered by Croyden and Hambly \cite{CroyHamb08}, who looked at a random homeomorphism for a deterministic fractal subeset of $\BR^2$ to the Brownian CRT. We present another such notion, framed analogously to Definition \ref{def:mass_struct} of mass-structural equivalence.

\begin{definition}\label{def:struct}
 Consider a rooted $\BR$-tree $(\T,d,r)$. A leaf $x\in \T$ is a \emph{discrete leaf} if there exists some branch point $y\in\T$ (its \emph{parent}) that separates $x$ from all other branch points. These discrete leaves, along with the branch points and the root $r$, comprise the set of \emph{structural points} of $(\T,d,r)$.
 
 Let $\mathscr{V}_i$ denote the set of structural points of a tree $(\T_i,d_i,r_i)$ for $i=1,2$. A \emph{structural isomophism} between these $\BR$-trees is a bijection $f\colon \mathscr{V}_1\to\mathscr{V}_2$ with the property that, for $x,y\in\mathscr{V}_1$, we have $x\in [[r_1,y]]_{\cT_1}$ if and only if $f(x)\in [[r_2,f(y)]]_{\cT_2}$.
 
 Two rooted $\BR$-trees are said to be \emph{structurally equivalent} if there exists a structural isomorphism from one to the other. It is straightforward to confirm that this is an equivalence relation.
\end{definition}

The following example illustrates the subtle distinction between the discrete leaves defined here and the isolated leaves of Definition \ref{def:special_pts}. We conjecture, and it should not be difficult to show, that replacing isolated leaves with discrete leaves of $\Span(p)$ in Definition \ref{def:special_pts} would yield an equivalent notion of mass-structural equivalence, but we will not prove this.

\begin{example}
 Let $(\T,\ell_1,0,p)$ be a Brownian CRT embedded in $\ell_1$ via the bead crushing construction discussed in the proof of Proposition \ref{prop:BIPT}. Let $x_1,x_2,\ldots$ be i.i.d.\ samples from $p$. For $n\ge 2$, let $\phi_n$ denote the linear transformation on $\ell_1$ that sends each coordinate vector $\Be_k$ to $\Be_{nk}$ for $k\ge 1$. Then, define
 \begin{equation*}
 \begin{split}
  \T_2 &:= \phi_2(\T) \cup \bigcup_{k\ge 1}\left(\phi_2(x_k) + [0,2^{-k}]\Be_{2k-1}\right),\\
  \T_3 &:= \phi_3(\T) \cup \bigcup_{k\ge 1}\left(\phi_3(x_k) + \left([0,2^{-k-1}]\Be_{3k-2} \cup [0,2^{-k-1}]\Be_{3k-1}\right)\right), \quad \text{and}\\
  \T_4 &:= \phi_4(\T) \cup \bigcup_{k\ge 1}\left(\phi_4(x_k) + \left([0,2^{-k-1}]\Be_{4k-3} \cup [0,2^{-k-2}]\Be_{4k-2}\right.\right.\\[-5pt]
  	&\hspace{3.2in} \left.\left.\cup\; [0,2^{-k-2}]\Be_{4k-1}\right)\right).
 \end{split}
 \end{equation*}
 In other words, $\T_2$ is formed by isometrically re-embedding $\T$ into the even coordinates in $\ell_1$, and then attaching new, macroscopic branches at each of the leaves $x_k$, $k\ge 1$; and $\T_3$ and $\T_4$ are correspondingly formed by attaching two or three new branches at each sampled leaf. For $n=2,3,4$, let $p_n$ denote the length measure on $\T_n\setminus \phi_n(\T)$, and consider $(\T_n,\ell_1,0,p_n)$ as a rooted, weighted $\BR$-tree. Then $\T_2 = \Span(p_2)$ and the leaves $\phi(x_k) + 2^{-k}\Be_{2k-1}$ are isolated leaves in the sense of Definition \ref{def:special_pts}, and correspondingly for $\T_3$ and $\T_4$. However, the newly added leaves in $\T_2$, in particular, are not ``discrete'' in the sense of Definition \ref{def:struct}, since leaves in a Brownian CRT do not have parent branch points but rather arise as limit points of branch points.
 
 If we did not include isolated leaves, like those in $\T_2$, $\T_3$, and $\T_4$, as special points, but otherwise left Definitions \ref{def:special_pts} and \ref{def:mass_struct} of special points and mass-structural equivalence as is, then $(\T_3,\ell_1,0,p_3)$ and $(\T_4,\ell_1,0,p_4)$ would be considered mass-structurally equivalent, and Theorems \ref{thm:MS_rep} and \ref{thm:MS_Theta} would fail.
 
 Now, define
 \begin{equation*}
 \begin{split}
  \T'_2 &:= \phi_2(\T) \cup \big(\phi((x_1\wedge x_2)_{\T}) + [0,1]\Be_{1}\big) \quad \text{and}\\
  \T'_3 &:= \phi_2(\T) \cup \big(\phi((x_1\wedge x_2)_{\T}) + ([0,1]\Be_{1} \cup [0,1]\Be_{3})\big).
 \end{split}
 \end{equation*}
 Consider $(\T'_2,\ell_1,0,\phi_2(p))$ and $(\T'_3,\ell_1,0,\phi_2(p))$. The newly added branches do not belong to $\Span(\phi_2(p))$, so their endpoints are not isolated leaves, in the sense of Definition \ref{def:special_pts}. But these endpoints are discrete leaves, in the sense of Definition \ref{def:struct}. Consequently, the two trees are mass-structurally equivalent to each other and to $(\T,\ell_1,0,p)$, but not structurally equivalent.
\end{example}

Structural equivalence may be an interesting notion of equivalence, but the ``underlying structure'' -- i.e.\ structural equivalence class -- as an object sacrifices much of what makes CRTs interesting. Reframing the result of Croyden and Hambly \cite{CroyHamb08} in the language of Definition \ref{def:struct}, we can construct a Brownian CRT such that its underlying structure is deterministic. Without either distances or masses to indicate relative ``sizes'' of components in a decomposition of the Brownian CRT, the randomness and much of the interesting fractal structure are lost.

\subsection{Directions for further study}\label{sec:problems}

(1) Introduce and study interesting families exchangeable hierarchies on $\BN$, or equivalently in light of Theorem \ref{thm:hierarchies_IP}, random IP trees, perhaps for use in applications such as nested topic models in machine learning \cite{NestedCRP,NHDP}. The three behaviors mentioned around the statement of Theorem \ref{thm:measure_decomp} -- macroscopic branching, broom-like explosion, and comb-like erosion -- cannot all be distinguished in the discrete regime, but the insight that all three can appear in scaling limits may aid in defining models for finite exchangeable random hierarchies.\medskip

(2) In connection with (1), do random IP trees arise as scaling limits of suitably metrized random discrete trees? Can we learn about the IP trees from this perspective? This may tie back to the perspective in Theorem \ref{thm:MS_Theta}, of IP trees as corresponding to e.i.g.\ hierarchies on $\BN$, and the latter being represented as projectively consistent sequences, as in Definition \ref{def:hierarchy}.\medskip

(3) Study the images of other rooted, weighted CRTs, for example those arising from bead-crushing constructions as in \cite{PitmWink09,RembWinkString} (including stable CRTs), under the map from a rooted, weighted $\BR$-tree to an isomorphism class of mass-structurally equivalent IP trees.\medskip

(4) Characterize mass-structural equivalence in terms of deformations or correspondences, in the sense described in \cite{MR2221786}. How can a tree be stretched, pruned, contracted, or otherwise modified without changing its mass-structure?\medskip

(5) Study notions of structural equivalence of CRTs that do not depend on either mass or quantified distance, such as those in Definition \ref{def:struct} or \cite{CroyHamb08}. Look at a space of $\BR$-tree structures. Consider random elements, metrize the space, etc..

\section*{Acknowledgments}

The author thanks Soumik Pal and Matthias Winkel for their support in this research and their helpful comments on a draft of this paper.

\bibliographystyle{plain}
\bibliography{IPtrees}

\end{document}